\documentclass[10pt]{article}
 \usepackage{amsthm,amsmath,amssymb, pgf, tikz, pgfpict2e}
 \usepackage{tocloft}
 \usepackage{graphicx}

\theoremstyle{plain}
\newtheorem{Thm}{Theorem}[section]
\newtheorem{Cor}[Thm]{Corollary}

\newtheorem{lemma}[Thm]{Lemma}
\newtheorem{Prop}[Thm]{Proposition}
\newtheorem{Def}[Thm]{Definition}

\newtheorem{remark}[Thm]{Remark}
\newtheorem{example}[Thm]{Example}

\newcommand{\abone}{[\alpha_1|\beta_1]}
\newcommand{\abtwo}{[\alpha_2|\beta_2]}
\newcommand{\abhat}{[\hat{\alpha}|\hat{\beta}]}

\newcommand{\ama}{\alpha_1 \wedge \alpha_2 }
\newcommand{\bmb}{\beta_1 \wedge \beta_2 }
\newcommand{\aja}{\alpha_1 \vee \alpha_2 }
\newcommand{\bjb}{\beta_1 \vee \beta_2 }

\newcommand{\talpha}{ \tilde{\alpha}}
\newcommand{\tbeta}{ \tilde{\beta}}

\newcommand{\tg}{ \tilde{G}}
\newcommand{\tk}{ \tilde{K}}
\newcommand{\ts}{ \tilde{S}}
\newcommand{\tpi}{ \tilde{\pi}}

\newcommand{\tkn}{\tau_k(n)}

\newcommand{\tknpr}{\tau'_k(n)}

\newcommand{\tkntknpr}{[\tkn|\tknpr]}

\newcommand{\tjn}{\tau_j(n)}

\newcommand{\tjnpr}{\tau'_j(n)}

\newcommand{\gmg}{\gamma_1 \wedge \gamma_2 }
\newcommand{\gjg}{\gamma_1 \vee \gamma_2 }

\newcommand{\ab}{[\alpha|\beta]}
\newcommand{\abpr}{[\alpha'|\beta']}

\date{\today  }
\title{Split graphs and Block Representations}

\author{Karen L. Collins\\
\small Dept. of Mathematics and Computer Science\\
\small Wesleyan University\\
\small Middletown CT 06459-0128\\
\small\tt kcollins@wesleyan.edu\
\and
Ann N. Trenk\thanks{
This work was supported by a grant from the Simons Foundation (\#426725, Ann Trenk). } \\
\small Department of Mathematics\\
\small Wellesley College\\
\small Wellesley MA 02481\\
\small\tt atrenk@wellesley.edu\
\and
Rebecca Whitman\\
\small Department of Mathematics\\
\small University of California, Berkeley\\
\small Berkeley, CA 94720\\
\small\tt  rebecca\_whitman@berkeley.edu\
}
\begin{document}
\maketitle

\begin{abstract}

In this paper, we study split graphs and related classes of graphs from the perspective of their sequence of vertex degrees and an associated lattice under majorization. Following the work of Merris \cite{Me03}, we define blocks\ $[\alpha(\pi)|\beta(\pi)]$, where $\pi$ is the degree sequence of a graph, and $\alpha(\pi)$ and $\beta(\pi)$ are sequences arising from $\pi$. We use the block representation $[\alpha(\pi)|\beta(\pi)]$ to characterize  membership in each of the following classes:  unbalanced split graphs, balanced split graphs, pseudo-split graphs, and three kinds of  Nordhaus-Gaddum graphs (defined in \cite{CoTr13,ChCoTr16}).  As in \cite{Me03}, we form a poset under the relation majorization
 in which the elements are the blocks $[\alpha(\pi)|\beta(\pi)]$ representing split graphs  with a fixed number of edges.  We partition this poset in several interesting ways  using what we call amphoras, and prove upward and downward closure results for blocks arising from different families of graphs.  Finally, we show that the poset becomes a lattice when a maximum and minimum element are added, and we prove properties of the meet and join of two blocks.

  \end{abstract}
  
  \smallskip
  

\bibliographystyle{plain}

\section{Introduction} 
\label{sec1}

A graph $G$ is a \emph{split graph} if its vertex set can be partitioned as $V(G) = K \cup S$ where the vertices in $K$ form a clique and the vertices in $S$ form a stable set, also known as an independent set.  F\"{o}ldes and Hammer  introduced split graphs in \cite{FoHa77}, and since then they have received considerable attention. Several authors have written book chapters on split graphs, including \cite{Go80,MaPe95, Me01} and, more  recently, \cite{CoTr21}. Hammer and Simeone \cite{HaSi81} proved that the class of split graphs is characterized by degree sequence. For any graph $G$, Merris \cite{Me03} proved that additional features of $G$ can be discerned from its degree sequence $\pi$ by dividing the Ferrers diagram $F(\pi)$  into two regions:  $A(\pi)$, consisting of the \emph{rows} of boxes starting at and including the main diagonal, and $B(\pi)$, consisting of the \emph{columns} of boxes below the main diagonal, and identifying   
the partition $\alpha(\pi)$ whose parts are the number of boxes in the rows of $A(\pi)$, and the partition $\beta(\pi)$ whose parts of the number of boxes in the columns of $B(\pi)$.  Both $\alpha(\pi)$ and $\beta(\pi)$ are partitions into distinct parts. Other authors  use  $\alpha(\pi)$ and $\beta(\pi)$ to determine whether a sequence $\pi$ is graphic, and whether it is the degree sequence of a split graph or a threshold graph.  These results are described in \cite{CoTr21}. 

A \emph{$KS$-partition} of a split graph  is a partition $V(G) = K \cup S$ where $K$ is a clique and $S$ is a stable set.  Let $\omega(G)$ denote the size of a largest clique in $G$ and $\alpha(G)$ denote the size of a largest stable set in $G$.  A $KS$-partition of $G$ is \emph{$K$-max} if $|K| = \omega(G)$ and \emph{$S$-max} if $|S| = \alpha(G)$.  A split graph is \emph{balanced} if it has a $KS$-partition that is simultaneously $K$-max and $S$-max, and \emph{unbalanced} otherwise. A degree sequence of a split graph is said to be balanced if its split graph is balanced and unbalanced if its split graph is unbalanced. We determine whether a split graph with degree sequence $\pi$ is balanced or unbalanced by examining  $\alpha(\pi)$ and $\beta(\pi)$ in Theorem~\ref{unbal-block-thm}, and we describe membership in several classes of graphs related to split graphs in terms of the ordered pairs, $[\alpha(\pi)|\beta(\pi)]$: threshold graphs, pseudo-split graphs and  Nordhaus-Gaddum graphs (NG-graphs are defined in \cite{CoTr13,ChCoTr16}).

The partitions of an integer  form a lattice under the majorization order; see Brylawski \cite{Br73}, and recent work by Ganter \cite{Ga22}, and we describe this in Section~\ref{sec-3}.  Ordered pairs of integer partitions form a poset and we consider the subposet $S$-$Block(n)$ corresponding to 
degree sequences of split graphs with $n$ edges.   We prove upward and downward closure results in $S$-$Block(n)$ for unbalanced and balanced pairs $[\alpha|\beta]$, 

determine  the meet and join of any two  elements of $S$-$Block(n)$, and characterize when these are balanced and unbalanced. We prove that $S$-$Block(n)$ is naturally partitioned into separate unbalanced and balanced subposets.  We call these  amphoras since each  is shaped like an ancient Greek amphora, a ceramic pot designed as an efficient shipping container with a pointed bottom and wider top.  These are  described in Theorem~\ref{partition-thm} and Corollary~\ref{cor-good}, and are shown for $n=10$ in Figures~\ref{fig-s-10} and \ref{fig-W}. Further, we show natural partitions into amphoras for threshold graphs, pseudo-split graphs and NG-graphs. 

The remainder of this paper is organized as follows.  In Section~2, we provide additional information about balanced and unbalanced split graphs and the Ferrers  diagrams of their degree sequences, as well as  background and preliminary results about threshold graphs, pseudo-split graphs, and three kinds of NG-graphs. Membership of a graph in each of these classes is determined by its degree sequence $\pi$, and in later sections we prove additional results about these classes using $\alpha(\pi)$ and $\beta(\pi)$.    Section~3 provides background on the  relation \emph{majorization} between partitions and  a resulting poset  $Dis_k(n)$, whose elements are the partitions of $n$ into  $k$ distinct parts. In addition, following the work of \cite{Me03},  we combine pairs of partitions into blocks,  and discuss majorization between blocks, resulting in the poset $S$-$Block(n)$.  
 In Section~4, we define an amphora in a poset,  and then 
  show how the poset $S$-$Block(n)$ is naturally partitioned  into amphoras using the categories of balanced, unbalanced and NG-graphs, and also
  show upward and downward closure results.  
    In Section~5, we extend this poset to a lattice by adding a minimum element $\hat{0}$ and a maximum element $\hat{1}$ and investigate properties of the meet and join of two blocks.  Finally, in Section~6, we show how pseudo-split graphs and NG-3 graphs with degree sequence $\pi$ can be characterized by $\alpha(\pi)$ and $\beta(\pi)$.

\section{Background and preliminaries}

\subsection{Ferrers diagrams, balanced and unbalanced split graphs}

Recall that a  split graph is \emph{balanced} if it has a $KS$-partition that is simultaneously $K$-max and $S$-max, and \emph{unbalanced} otherwise.  
 For example, the graph $P_4$ is a balanced split graph while $K_{1,3}$ is unbalanced.   For any  $K$-max partition of an unbalanced split graph, there is a vertex $k \in K$ so that $S \cup \{k\}$ is a stable set, and for any $S$-max partition of an unbalanced split graph, there is a vertex $s \in S$ so that $K \cup \{s\}$ is a clique.  We call such vertices, $k$ or $s$,  \emph{swing vertices}.  For  more details, see \cite{CoTr21}.

Throughout this paper we will assume that any  partition of integers   is written in 
     non-increasing order, so when we write $\pi = (d_1,d_2, \ldots, d_r)$,  it follows that
      $d_1  \ge d_2 \ge \cdots \ge d_r$.   We say the \emph{length} of $\pi$ is $r$ and write  len$(\pi) =r$.  The \emph{Ferrers diagram} $F(\pi)$ is the diagram with $r$ left-justified rows and $d_i$ boxes in row $i$ for $1 \le i \le r$.         We are primarily interested in Ferrers diagrams $F(\pi)$ when $\pi$ is the degree sequence of a graph.  
   We will  ignore isolated vertices in graphs, so our degree sequences will consist of positive integers. 
     Figure~\ref{fig-central-rect}
     shows the Ferrers diagrams  for $\pi_1$ and $\pi_2$, where  $\pi_1 = (6,5,2,2,2,1,1,1)$ and $\pi_2 = (6,4,3,2,2,1,1,1)$.

 A fundamental parameter in recognizing split graphs from degree sequences is the mark, also known as the \emph{modified Durfee number}.  
   If $\pi = (d_1,d_2, \ldots, d_r)$,  the \emph{mark} of $\pi$ is defined as  $m(\pi) = \max \{i:d_i \ge i-1\}$ and we write $m$ for $m(\pi)$ when there is no ambiguity.  Since $d_m \ge m-1$, there is a box in $F(\pi)$ in  row $m$, column $m-1$ and we define the \emph{central rectangle} to be the $m \times (m-1)$ rectangle in the upper left of $F(\pi)$.   By the definition of $m$, we know $d_{m+1} < m$, so the central rectangle  cannot be extended to an $(m+1) \times m$ rectangle
   and thus it provides a visual way to identify $m(\pi)$ from $F(\pi)$.  For $\pi_1 $  and $\pi_2 $ given above,  we have $m(\pi_1) = m(\pi_2)= 3$ and the $3 \times 2$ central rectangles are shown in Figure~\ref{fig-central-rect}. 
  
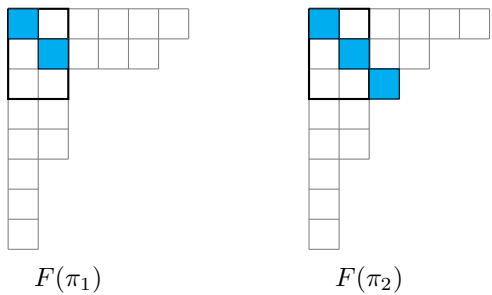
\begin{figure}\begin{center} 
\begin{tikzpicture}[scale=.8]


\draw[gray] (0,0)--(0,4);
\draw[gray]  (.5,0)--(.5,4);
\draw[gray]  (1,1.5)--(1,4);
\draw[gray]  (1.5,3)--(1.5,4);
\draw[gray]  (2,3)--(2,4);
\draw[gray]  (2.5,3)--(2.5,4);
\draw[gray]  (3,3.5)--(3,4);

\draw[gray]  (0,4)--(3,4);
\draw[gray]  (0,3.5)--(3,3.5);
\draw[gray]  (0,3)--(2.5,3);
\draw[gray]  (0,2.5)--(1,2.5);
\draw[gray]  (0,2)--(1,2);
\draw[gray]  (0,1.5)--(1,1.5);
\draw[gray]  (0,1)--(.5,1);
\draw[gray]  (0,0)--(.5,0);
\draw[gray]  (0,.5)--(.5,.5);

\draw[thick, black] (0,4)--(1,4)--(1,2.5)--(0,2.5)--(0,4);
\draw [fill=cyan] (0,3.5) rectangle (.5,4);
\draw [fill=cyan] (.5,3) rectangle (1,3.5);

\node(0) at (1,-.5) {$F(\pi_1)$};

\draw[gray] (5,0)--(5,4);
\draw[gray]  (5.5,0)--(5.5,4);
\draw[gray]  (6,1.5)--(6,4);
\draw[gray]  (6.5,2.5)--(6.5,4);
\draw[gray]  (7,3)--(7,4);
\draw[gray]  (7.5,3.5)--(7.5,4);
\draw[gray]  (8,3.5)--(8,4);

\draw[gray]  (5,4)--(8,4);
\draw[gray]  (5,3.5)--(8,3.5);
\draw[gray]  (5,3)--(7,3);
\draw[gray]  (5,2.5)--(6.5,2.5);
\draw[gray]  (5,2)--(6,2);
\draw[gray]  (5,1.5)--(6,1.5);
\draw[gray]  (5,1)--(5.5,1);
\draw[gray]  (5,0)--(5.5,0);
\draw[gray]  (5,.5)--(5.5,.5);

\draw[thick, black] (5,4)--(6,4)--(6,2.5)--(5,2.5)--(5,4);
\draw [fill=cyan] (5,3.5) rectangle (5.5,4);
\draw [fill=cyan] (5.5,3) rectangle (6,3.5);
\draw [fill=cyan] (6,2.5) rectangle (6.5,3);

\node(0) at (6, -.5) {$F(\pi_2)$};

\end{tikzpicture}\end{center}

\caption{The Ferrers diagrams and illustration of the central rectangles for $\pi_1 =(6,5,2,2,2,1,1,1)$ and $\pi_2= (6,4,3,2,2,1,1,1)$ }
 \label{fig-central-rect}
\end{figure}

Additional features of a  graph  with degree sequence $\pi$ can be discerned by dividing $F(\pi)$ into two regions:  $A(\pi)$, consisting of the \emph{rows} of boxes starting at and including the main diagonal, and $B(\pi)$, consisting of the \emph{columns} of boxes below the main diagonal.     These regions  in turn lead to two partitions derived from $\pi$, the partition $\alpha(\pi)$ whose parts are the number of boxes in the rows of $A(\pi)$, and  the partition $\beta(\pi)$ whose parts of the number of boxes in the columns of $B(\pi)$.    In Figure~\ref{fig-central-rect}, the boxes in on the main diagonal are shaded, and  $\alpha(\pi_1) = (6,4)$,  $\beta(\pi_1) = (7,3) $,  $\alpha(\pi_2) = (6,3,1)$,  and  $\beta(\pi_2) = (7,3) $.

By construction,   $\alpha(\pi)$ and  $\beta(\pi)$ are partitions into distinct parts.    As the following theorem shows, the length of $\beta(\pi)$  always equals $m(\pi) - 1$ while the 
   length of $\alpha(\pi)$ depends on whether or not there is a box on the main  diagonal of row $m$ of $F(\pi)$.  Both possibilities are illustrated in Figure~\ref{fig-central-rect};   len$(\alpha(\pi_1)) = m(\pi_1) -1$ and len$(\alpha(\pi_2)) = m(\pi_2) $. 
 Theorem~\ref{length-alpha} appears in \cite{CoTr21}.
 
   \begin{Thm} 
 If  $\pi$ is a partition  and $m = m(\pi)$ then    ${\rm len}(\beta(\pi)) = m-1$.  Furthermore, either ${\rm len}(\alpha(\pi)) = { \rm len}(\beta(\pi)) = m - 1$, or ${\rm len}(\alpha(\pi)) = {\rm len}(\beta(\pi)) + 1 = m.$
 \label{length-alpha}
 \end{Thm}
 
 \begin{proof}
  By the definition of the mark $m$, we know $d_m \ge m-1$ and $d_{m+1} < m$, so there is a box in row $m$,  column $m-1$ of $F(\pi)$ and no box in row $m+1$, column $m$ of $F(\pi)$.  Thus ${\rm len}(\beta(\pi)) = m-1$ and  ${\rm len}(\alpha(\pi)) \le m$.    If there is a box in row $m$, column $m$  of $F(\pi)$, then ${\rm len}(\alpha(\pi)) = {\rm len}(\beta(\pi)) + 1 = m$, and otherwise, ${\rm len}(\alpha(\pi)) = { \rm len}(\beta(\pi)) = m - 1$.
 \end{proof}

If $\pi$ is the degree sequence of a split graph $G$, then the two possibilities described in Theorem~\ref{length-alpha}
also determine whether $G$ is balanced or unbalanced.

 \begin{Thm}  If $G$ is a split graph   with degree sequence $\pi$ then $G$ is unbalanced when $\alpha(\pi)$ and $\beta(\pi)$ have the same length and balanced otherwise.
  \label{unbal-block-thm}
 \end{Thm}

 \begin{proof}   
 Let $G$ be a split graph with degree sequence $\pi = (d_1,d_2, \ldots, d_r)$ and let $m = m(\pi)$.  In \cite{ChCoTr16},  it is shown that a  split graph $G$ is unbalanced if $d_m = m-1$ and balanced if $d_m > m-1$, and  a  short and direct proof of this appears in \cite{CoTr21}.   Hence if  $G$ is unbalanced, then $d_m = m-1$, so len$(\alpha(\pi) ) = m-1 = $ len$(\beta(\pi))$ using Theorem~\ref{length-alpha}.  If $G$ is balanced, then $d_m > m-1$, so len$(\alpha(\pi) ) > m-1 = $ len$(\beta(\pi))$.  
 \end{proof}

 As a result of Theorem~\ref{unbal-block-thm}, we can determine whether a split graph  with degree sequence $\pi$ is balanced or unbalanced by examining  $\alpha(\pi)$ and $\beta(\pi)$.   
Analogous results appear in Theorems~\ref{NG-block-thm} and  \ref{NG3-thm} and Corollary~\ref{pseudo-charac}.

\subsection{Threshold graphs, NG-graphs, and pseudo-split graphs}
 \label{sec-2}

Threshold graphs play a special role in the  poset $S$-$Block(n)$  which we will define in Section~\ref{sec-4}.
The class of threshold graphs was introduced by Chv\'{a}tal and Hammer 
 and studied extensively in \cite{MaPe95}.  
A graph $G=(V,E)$ is a \emph{threshold graph} if there exist a threshold $t > 0$ and a positive weight $a_i$ assigned to each vertex $v_i\in V$, so that $S\subseteq V$ is a stable set of $G$ if and only if $\sum_{i \in S}a_i \le t$.

Split graphs and threshold graphs have similar forbidden subgraph characterizations.  A graph $G$ is a split graph if and only if it   contains none   of $2K_2$, $C_4$, or $C_5$ as an  induced subgraph \cite{FoHa77}.  A graph $G$ is a threshold graph if and only if it contains none of 
$2K_2$, $C_4$,  or $P_4$ as an induced subgraph \cite{ChHa77}.  Since $P_4$ is an induced subgraph of $C_5$, it immediately follows that  all threshold graphs are split graphs.   It is natural to ask whether they are balanced or unbalanced.  Theorem~\ref{P4-thm} shows that all balanced split graphs contain an induced $P_4$, and in Corollary~\ref{thresh-unbal-cor},
we conclude that threshold graphs are unbalanced split graphs.

\begin{Thm}
Balanced split graphs contain an induced $P_4$.
\label{P4-thm}
\end{Thm}

\begin{proof}   Let $G$ be a balanced split graph and $V = K \cup S$ be its unique $KS$-partition.    For vertices $v \in K$, we denote the set of vertices in $S$ that are adjacent to $v$ by  $N_S(v)$.    Let $x$ be a vertex of $K$ with minimum degree.  If $N_S(x) = \emptyset$ then $x$ is a swing vertex, a contradiction because $G$ is balanced.  Thus $N_S(x) \neq \emptyset$ and there exists $z \in N_S(x)$.  Since $z$ is not a swing vertex, there exists a vertex $y \in K$ that is not adjacent to $z$, and thus $z \in N_S(x) - N_S(y)$.  However,     $deg(x) \le deg(y)$, so  there exists $w \in N_S(y) - N_S(x)$.  Now vertices  $x,y,z,w$ induce a path $P_4$ in $G$ as desired. 
\end{proof}

\begin{Cor}
Threshold graphs are unbalanced split graphs.
\label{thresh-unbal-cor}
\end{Cor}

\begin{proof}
As noted above, it follows from the forbidden graph characterizations that all threshold graphs are split graphs.  By 
  Theorem~\ref{P4-thm}, all balanced split graphs contain an induced $P_4$, and therefore threshold graphs cannot be balanced split graphs, so they must be unbalanced split graphs.
\end{proof}

 An alternate proof of Corollary~\ref{thresh-unbal-cor} can be obtained by combining 
  Theorems~\ref{unbal-block-thm}, \ref{split-S-block-thm} and  \ref{thresh-S-block-thm}.
  As we will see in Section~\ref{sec-4}, threshold graphs play a fundamental role as the maximal elements  in the split block poset.  

A well-known  theorem by Nordhaus and Gaddum \cite{NoGa56}  relates the chromatic number of any  graph $G$  and its complement $\overline{G}$  to the  number of vertices:     
$$ 2 \sqrt{|V(G)|}   \leq \chi(G) + \chi(\bar{G}) \leq  |V(G)| + 1.$$   We call $G$ a {\it Nordhaus-Gaddum graph} or {\it NG-graph}  if $ \chi(G) + \chi(\bar{G})  =   |V(G)| + 1.$
Finck \cite{Fi66}  and Starr and Turner \cite{StTu08} provide  two different  characterizations of NG-graphs.   Collins and Trenk  \cite{CoTr13}  define  the ABC-partition of a graph and characterize  NG-graphs in terms of this partition.

\begin{Def}\rm 
For a graph $G$, the  \emph{ABC-partition} of $V(G)$  (or of $G$) is  

 $A_G = \{v \in V(G): deg(v) = \chi(G) -1\}$
 
 $B_G = \{v \in V(G): deg(v) > \chi(G) -1\}$
 
 $C_G = \{v \in V(G): deg(v) < \chi(G) -1\}$.   
 
 When it is unambiguous, we write $A=A_G$, $B=B_G$, $C=C_G$.
  
\label{ABC-def}
\end{Def}

 For a graph $G$ and any subset of vertices $V' \subseteq V(G)$, let $G[V']$ denote the subgraph induced  in $G$ by $V'$.  The following theorem  from \cite{CoTr13} characterizes NG-graphs using ABC-partitions.

\begin{Thm}  {\rm (Collins and Trenk \cite{CoTr13})}
 A graph $G$  is an NG-graph if and only if its ABC-partition satisfies
  
 (i) $A \neq \emptyset$ and  $G[A]$ is a clique, a stable set, or a 5-cycle
 
 (ii) $G[B]$ is a clique
 
 (iii)  $G[C]$ is a stable set
 
 (iv)  $uv \in E(G)$ for all $u \in A$, $v \in B$ 
 
 (v)  $uw \not\in E(G)$ for all $u \in A$, $w \in C$.
 
\label{NG-charac}
\end{Thm}

By (i) of Theorem~\ref{NG-charac}, there are  three possible forms of an NG-graph, and they are defined as follows.  

\begin{Def} \rm A graph $G$  is an NG-1 graph if $G[A]$ is a clique, an  NG-2 graph if  $G[A]$ is a stable set, and an NG-3 graph if  $G[A]$ is a 5-cycle.  \label{NG-123-def}
\end{Def}

It is shown in \cite{ChCoTr16} that if $G$ is  any NG-graph then $\chi(G) = m(G)$.  This leads to the following remark.

\begin{remark}
If $G$ is an NG-graph, then we may replace $\chi(G)$ by $m(G)$ in its ABC-partition.
\label{m-ABC-rem}
\end{remark}

Note that  NG-3 graphs are not split graphs because they contain a 5-cycle.
However, NG-1 graphs (resp. NG-2 graphs) are split graphs with $KS$-partition $K = A \cup B$, $S = C$ (resp. $K = B$, \ $S = A \cup C$).    Indeed, they are \emph{unbalanced} split graphs  because the vertices in the non-empty $A_G$ are swing vertices by conditions    (iv) and (v) of Theorem~\ref{NG-charac}.  Conversely, it is shown in \cite{ChCoTr16} and, later, in Corollary~\ref{NG-cor}
 that every unbalanced split graph is an NG-1 graph or an NG-2 graph (or both).  We record this equivalence below in Proposition~\ref{unbal-NG-prop}.
\begin{Prop}
The class of unbalanced split graphs is the union of the classes of NG-1 graphs and NG-2 graphs.
\label{unbal-NG-prop}
\end{Prop}

A graph is a \emph{pseudo-split graph} if it contains neither $C_4$ nor $2K_2$ as an induced subgraph.  This class was first studied by  Bl\'{a}zsik et. al.  in \cite{BlHuPlTu93}  and its connection to NG-graphs was proven in the following theorem in \cite{ChCoTr16}.

  \begin{Thm}   The class of pseudo-split graphs equals the following disjoint unions.

\vspace{.2cm}
{\rm (i)} $ \{$NG-graphs$\}  \ \dot\cup  \ \{$balanced split graphs$\}$

{\rm (ii) }$ \{$split graphs$\}  \ \dot\cup  \ \{$NG-3 graphs$\}$
  
  
  
  \label{pseudo-split-union}
  \end{Thm}
  
  In section \ref{sec-4}  we characterize  all three types of NG-graphs in terms of the $\alpha(\pi)$ and $\beta(\pi)$ of their degree sequence $\pi$, and do the same for pseudo-split graphs in Section~\ref{sec-6}.

 \section{Majorization in $Dis_k(n)$ and $S$-$Block(n)$}
 \label{sec-3}
 
 In this section  we provide background  definitions and results about majorization  and the resulting posets $Dis_k(n)$ and $S$-$Block(n)$. Both of these are related to the lattice of partitions of an integer under majorization, introduced by Brylawski \cite{Br73}  and further developed by Ganter \cite{Ga22}, who provided an algorithm for determining meets and joins, enumerated meet- and join-irreducible elements, and computed a number of lattice statistics. Le and Phan \cite{LePh07} proved that the poset of partitions of an integer into distinct parts is a lattice. In Section~\ref{dis-sec} we primarily focus on  the lattice of partitions of $n$ into distinct parts and in Section~\ref{block-sec} we form blocks containing two of these partitions to construct  the poset $S$-$Block(n)$. 
  
  \subsection{The lattices $Dis(n)$ and $Dis_k(n)$}
  \label{dis-sec}
  
    We first define the well-known relations of  \emph{weak majorization} and \emph{majorization} (also known as \emph{dominance}).
  
   \begin{Def}
 {\rm
 Let $\alpha$ be a     partition of $n_1$ and $\beta$ be    a partition of $n_2$,  where $\alpha = (\alpha_1,\alpha_2, \ldots, \alpha_s)$ and $\beta = (\beta_1,\beta_2, \ldots, \beta_t)$.
  We say \emph{ $\beta$  weakly majorizes $\alpha$}, denoted $\beta \succeq_w \alpha$, if  $n_2 \ge n_1$ and 
 $$\sum_{i=1}^k \beta_i  \ge \sum_{i=1}^k \alpha_i  \ \ \hbox{   for all } k \le \min\{s,t\}.$$
 Furthermore, $\beta $ \emph{majorizes}  $\alpha$, denoted $\beta \succeq \alpha$ if in addition $n_1 = n_2$. 
  If $\beta \succeq \alpha$ and $\beta \neq \alpha$, we write $\beta \succ \alpha$, and similarly for weak majorization.
 }
 \end{Def}
 
 As an example, note that $(7,4,1) \succeq_w (5,3,2,1)$ and  $(7,4,1) \succeq (6,3,2,1)$.  
 
   It is not hard to check that the set of partitions of $n$ into  parts that are positive and distinct  forms a poset under majorization.  We denote this poset by $Dis(n)$ and  denote by 
  $Dis_k(n)$  the subposet of $Dis(n)$ of partitions into \emph{exactly} $k$ parts.   It is well-known that the set of partitions of an integer $n$ form a lattice under the majorization order \cite{Br73} and 
 $Dis(n)$ is shown to be a lattice in \cite{LePh07}.  We  record  the latter below.

  \begin{Thm} 
The poset $Dis(n)$ is a lattice. 
\end{Thm}
 
In Proposition~\ref{dis-k-n-lattice-prop} we show  the poset $Dis_k(n)$ is also a lattice.  Our proof is a consequence of 
  identifying  minimum and maximum elements $\tau_k(n)$ and $\tau'_k(n)$  of $Dis_k(n)$.  The partitions $\tau_k(n)$ and $\tau'_k(n)$ will also play a fundamental role in later sections.

 Figure~\ref{fig-dis-10-12}   shows $Dis(10)$, $Dis(12)$, and a sublattice of $Dis(14)$. In Section~\ref{block-sec} we will use pairs of elements of $Dis(n)$ to form  a new poset, $S$-$Block(n)$.  The poset $S$-$Block(10)$ is shown in Figure~\ref{fig-s-10}.
 Observe that lower elements of $Dis(n)$ have longer lengths, as we make precise in the next lemma.\\

\small

\begin{figure}[h]
\begin{center}
\begin{tikzpicture}[scale=.65]
\draw[thick] (1,0)--(1,3);
\draw[thick] (1,5)--(1,8);
\draw[thick] (1,3)--(2,4)--(1,5)--(0,4)--(1,3);
\filldraw[blue]

(1,0) circle [radius=3pt]
(1,1) circle [radius=3pt]
(1,2) circle [radius=3pt]
(1,3) circle [radius=3pt]
(0,4) circle [radius=3pt]
(2,4) circle [radius=3pt]
(1,5) circle [radius=3pt]
(1,6) circle [radius=3pt]
(1,7) circle [radius=3pt]
(1,8) circle [radius=3pt];

\node at (2,0) {(4,3,2,1)};
\node at (1.8,.97) {(5,3,2)};
\node at (1.8,1.97) {(5,4,1)};
\node at (1.8,2.95) {(6,3,1)};
\node at (1.7,5.1) {(7,3)};
\node at (1.7,6) {(8,2)};
\node at (1.7,7) {(9,1)};
\node at (1.6,8) {(10)};
\node at (2.6,4) {(6,4)};
\node at (-.8,4) {(7,2,1)};

\node at (1,-1) {(a)};

\draw[thick] (6,0)--(5,1)--(6,2)--(5,3)--(6,4)--(5,4.8)--(5,6)--(6,7)--(6,10);
\draw[thick] (6,0)--(7,1)--(6,2)--(7,3)--(6,4)--(7,5)--(7,6.2)--(6,7);
\draw[thick] (5,4.8)--(7,6.2);

\filldraw[blue]

(6,0) circle [radius=3pt]
(5,1) circle [radius=3pt]
(7,1) circle [radius=3pt]
(6,2) circle [radius=3pt]
(5,3) circle [radius=3pt]
(7,3) circle [radius=3pt]
(6,4) circle [radius=3pt]
(5,4.8) circle [radius=3pt]
(7,5) circle [radius=3pt]
(5,6) circle [radius=3pt]
(7,6.2) circle [radius=3pt]
(6,7) circle [radius=3pt]
(6,8) circle [radius=3pt]
(6,9) circle [radius=3pt]
(6,10) circle [radius=3pt];

\node at (7.1,0) {(5,4,2,1)};
\node at (7.8,1) {(5,4,3)};
\node at (6.9,2) {(6,4,2)};
\node at (7.8,3) {(6,5,1)};
\node at (6.9,4) {(7,4,1)};
\node at (4,1) {(6,3,2,1)};
\node at (4.2,3) {(7,3,2)};
\node at (4.2,4.8) {(8,3,1)};
\node at (4.2,6) {(9,2,1)};
\node at (7.6,5.1) {(7,5)};
\node at (7.6,6.2) {(8,4)};
\node at (6.7,7.1) {(9,3)};
\node at (6.8,8) {(10,2)};
\node at (6.8,9) {(11,1)};
\node at (6.6,10) {(12)};

\node at (6,-1) {(b)};

\filldraw[blue]

(11,2.5) circle [radius=3pt]
(10.3,4) circle [radius=3pt]
(11.8,3.5) circle [radius=3pt]
(11.8,4.5) circle [radius=3pt]
(11,5.5) circle [radius=3pt];

\node at (12.1,2.4) {(7,4,2,1)};
\node at (12.7,3.5) {(7,4,3)};
\node at (12.7,4.5) {(7,5,2)};
\node at (9.2,4) {(8,3,2,1)};
\node at (11.9,5.6) {(8,4,2)};
\node at (11,-1) {(c)};

\draw[thick] (11,2.5)--(10.3,4)--(11,5.5)--(11.8,4.5)--(11.8,3.5)--(11,2.5);

\end{tikzpicture}

\caption{The Hasse Diagrams for  (a)  $Dis(10)$, (b)  $Dis(12)$, and  (c) a portion of $Dis(14)$.   }

\label{fig-dis-10-12}
\end{center}
\end{figure}
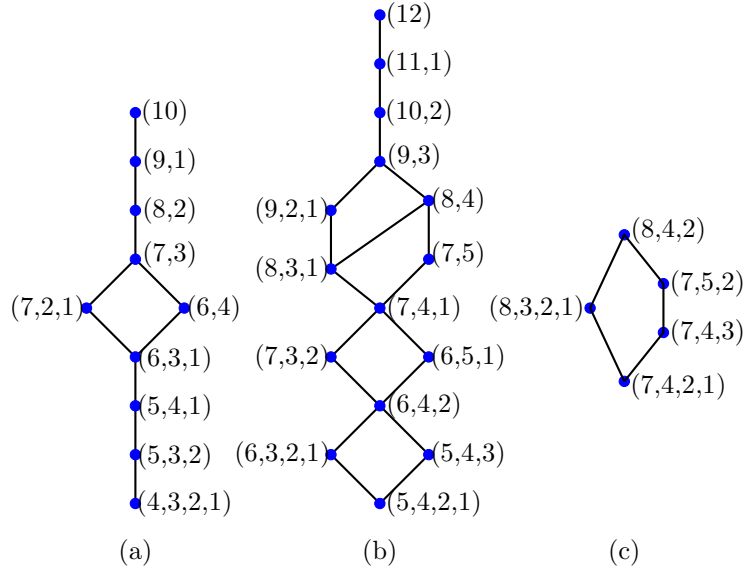

\normalsize

\begin{lemma}
If $\gamma_1, \gamma_2 \in Dis(n)$ and $\gamma_1 \succeq \gamma_2$ then len$(\gamma_1) \le $ len$(\gamma_2)$.
 \label{length-lemma}
\end{lemma}

\begin{proof}
Let $\gamma_1 = (b_1,b_2, \ldots, b_r)$ for integers  $b_1 > b_2 > \cdots > b_r > 0$, and let 
$\gamma_2 = (c_1,c_2, \ldots, c_s)$ for integers  $c_1 > c_2 > \cdots > c_s > 0$.  We wish to show $r \le s$.  For a contradiction, assume $r > s$.  Then $\sum_{i = s+1}^r b_i > 0$ so
 $$\sum_{i = 1}^s c_i  = n = \sum_{i = 1}^r b_i  > \sum_{i = 1}^s b_i $$
 contradicting the hypothesis that $\gamma_1 \succeq \gamma_2$ when we sum $s$ terms.
\end{proof}

  We next show there exist maximum and minimum elements of $Dis_k(n)$ when $Dis_k(n) \neq \emptyset$.
 If $\alpha = (a_1, a_2, \ldots, a_k) \in Dis_k(n)$, then $a_k \ge 1, a_{k-1} \ge 2, \ldots, a_2 \ge k-1$, and  $a_1 \ge k$; thus if $Dis_k(n)$ is not empty, we have $n \ge 1 + 2 + \cdots + k = \binom{k+1}{2}$.  We will see that in this case,  $Dis_k(n)$ has a maximum element $\tau'_k(n) $, given in Definition~\ref{tau-max-def}, in which the first part is as large as possible.  Likewise, $Dis_k(n)$ has a minimum element $\tau_k(n) $, defined in Definition~\ref{tau-min-def}, in which the parts are as equal in size as possible.  For $n=10$ and $k=3$, we have $\tau'_3(10) = (7,2,1)$ and $\tau_3(10) = (5,3,2)$.

      \begin{Def} {\rm For $n \ge \binom{k+1}{2}$, let  $\tau'_k(n) = (n-\binom{k}{2}, k-1, k-2, \ldots, 3,2,1)$.  }

    \label{tau-max-def}
  \end{Def}
  
     \begin{lemma}
  If $n \ge \binom{k+1}{2}$ then   $\tknpr $ is the  maximum element of $Dis_k(n)$.

    \label{tau-pr-max-lem}
  \end{lemma}
  
  \begin{proof}
 Let $\tau'_k(n) = (t'_1, t'_2, \ldots, t'_k)$ 
      and note that $\tau'_k(n) \in Dis_k(n)$, since $t_1'=n-\binom{k}{2}\geq \binom{k+1}{2} -\binom{k}{2}=k$. 
  We let $\alpha = (a_1, a_2, a_3, \ldots, a_k) \in Dis_{k}(n)$ and show  $\tknpr \succeq \alpha$.     Since the parts of $\alpha$ are decreasing in size, we have $a_k\ge 1=t'_k, \ a_{k-1} \ge 2=t'_{k-1},  \  \ldots, \  a_{2} \ge k-1=t'_2$, and $a_1 = n - (a_2 + a_3 + \cdots + a_k)$. 
Then for all $r$ with $1 \le r \le k$, 
 \[\sum_{i=1}^r a_i = \sum_{i=1}^k a_i -\sum_{i=r+1}^k a_i  =n -\sum_{i=r+1}^k a_i  \   \leq  \ 
 \sum_{i=1}^k t'_i -\sum_{i=r+1}^k t'_i=\sum_{i=1}^r t'_i.\] 
 So $\alpha\preceq \tau'_k(n) .$   \end{proof}

 In the next lemma we prove that the elements $\tau'_k(n) $ form a chain in $Dis(n)$.
 
 \begin{lemma}
 Suppose $n$, $j$  and $k$ are positive integers with $n \ge \binom{k+1}{2}$. 
 If $k > j$ then $\tau'_k(n) \preceq \tau'_j(n)$.
 \label{tau-pr-jk-lem}
 \end{lemma}
  
  \begin{proof}
  For $k > j$, the last $j-1$ parts of $\tau'_k(n)$ and $\tau'_j(n)$ are equal, and the sum of the first $k-j+1$ parts of $\tau'_k(n)$ equals the first part of $\tau'_j(n)$. Thus $\tau'_j(n) \succeq \tau'_k(n)$.    
  \end{proof}
   
  We are now ready to define $\tau_k(n)$ and prove it is the  minimum element in  $Dis_k(n)$.    Intuitively, we can think of constructing $\tkn$ by first assigning 1 to the smallest part, 2 to the next smallest part, etc. and then dividing the remaining $n - \binom{k+1}{2}$ as evenly as possible among the $k$ parts, starting at the largest parts.
  
  \begin{Def}
  {\rm
  Suppose  $n$ and $k$ are positive integers with $n \ge \binom{k+1}{2}$.    Let $n_k = n- \binom{k+1}{2} $ and write $n_k = qk+r$ where $q$ and $r$ are non-negative integers and $0 \le r < k$.   Let $\hat{\tau}_k(n_k) = (\hat{s}_1,\hat{s}_2, \ldots, \hat{s}_k)$ where $\hat{s}_i =   q+1$  for $ 1 \le i  \le r$, and  $\hat{s}_i = q$ for $r < i \le k$.   Finally, let $\tkn = (\hat{s}_1,\hat{s}_2, \ldots, \hat{s}_k)+ (k, k-1, k-2, \ldots, 1)$.
  
  }
    \label{tau-min-def}
  \end{Def}

 We give an example of the construction for $n = 20$ and $k = 3$.  In this case, $n_3 = 20 - \binom{4}{2} = 14$, we write $14 = 4 \cdot 3 + 2$ and get  $\hat{\tau}_3(14) = (5,5,4)$.  Thus $\tau_3(20) = (5,5,4) + (3,2,1) = (8,7,5)$.  Alternatively, $\tkn$ is the unique  element $(t_1, t_2, t_3, \ldots, t_k)$ of $Dis_k(n)$ for which $t_1 - t_k \le k$. It is not hard to show the following results about $\hat{\tau}_k(n_k)$, which we record in a remark.
 
\begin{remark}  
The unique minimum partition of $n_k$ into at most $k$ parts under majorization is $\hat{\tau}_k(n_k)$, and $\hat{\tau}_k(n_k)\preceq \hat{\tau}_{k-1}(n_k)$. 
\label{tau-rem}

\end{remark}
    
  \begin{lemma}
  If  $n \ge \binom{k+1}{2}$ then  $\tau_k(n) $ is the unique minimum element of $Dis_k(n)$.

  \label{tau-min-lem}
    
  \end{lemma}
  
\begin{proof}
We let  $\alpha \in Dis_k(n)$  and show $\tau_k(n) \preceq \alpha$.  
  Write $\alpha = (a_1, a_2, a_3, \ldots, a_k) $, and let 
 $a'_1 = a_1-k$,  $a'_2 = a_2-(k-1)$, \ldots, $a'_k = a_k-1$.  By construction, $a'_1 \ge a'_2 \ge \cdots \ge a'_k\ge 0$ and the $a_i'$  sum to $n_k=n-\binom{k+1}{2}$.  
The non-zero $a'_i$ form a partition $\hat{\alpha}$ of $n_k$ into at most $k$ parts. Since $\hat{\tau}_k(n_k)$, defined in Definition~\ref{tau-min-def},  is the unique minimum partition of $n_k$ into at most $k$ parts, then  $\hat{\tau}_k(n_k) \preceq  \hat{\alpha}$.
  It follows that $\tau_k(n) \preceq \alpha$, and $\tau_k(n)$ is the unique minimum element of $ Dis_k(n)$. 
\end{proof}
  
  The next lemma is analogous to Lemma~\ref{tau-pr-jk-lem} and  proves that  the elements $\tau'_k(n) $  also form a chain in $Dis(n)$.

 \begin{lemma}
 Suppose $n$, $j$, and $k$ are positive integers with $n \ge \binom{k+1}{2}$.  If $k > j$,
  then $\tau_k(n) \preceq \tau_j(n)$.
 \label{tau-jk-lem}
 \end{lemma}
 
 \begin{proof} 
 We will show that $\tau_k(n) \preceq \tau_{k-1}(n)$; the result in the lemma then follows by transitivity.  First note that $n = n_{k} + \binom{k+1}{2}$ and $n = n_{k-1} + \binom{k}{2}$, so $n_{k-1} = n_k + k$.  Then using Definition~\ref{tau-min-def} and Remark~\ref{tau-rem} we have,
 \begin{eqnarray*}
  \tau_{k-1}(n)  &=& \hat{\tau}_{k-1}(n_{k-1}) + (k-1, k-2, \ldots ,2, 1)\\
  & =&  \hat{\tau}_{k-1}(n_{k-1} - (k-1)) + (k, k-1, \ldots, 3,2) \\
  &=&
\hat{\tau}_{k-1}(n_{k} + 1) + (k, k-1, \ldots, 3, 2)  \\
  &\succeq & \hat{\tau}_{k}(n_{k}+ 1) + (k, k-1, \ldots, 3,2,0) \\
  &\succeq & \hat{\tau}_{k}(n_{k}) + (k, k-1, \ldots,3, 2, 1) =  \tau_{k}(n).
 \end{eqnarray*}
 \end{proof}

 The next result specifies the lengths  of the meet ($\wedge$) and join ($\vee$) of two elements in  the lattice $Dis(n)$.

  \begin{Prop}
  If     $ \gamma_1 \in Dis_k(n)$ and $ \gamma_2 \in Dis_j(n)$ then len$(\gjg) = \min \{ j,k\}$ and len$(\gmg) = \max \{ j,k\}$. 
Moreover, if $ \gamma_1, \gamma_2 \in Dis_k(n)$ then len$(\gjg) =  $ len($\gmg) = k$.
  \label{mj-length-prop}
  \end{Prop}
  
  \begin{proof}
    The last sentence follows immediately, so we prove the first sentence of the proposition.
  Suppose  $\gamma_1 \in Dis_k(n)$ and $ \gamma_2 \in Dis_j(n)$ where $j \le k$, and let $\tkn, \tjn, $ $\tknpr, \tjnpr$ be  the partitions in $Dis(n)$ defined in Definitions~\ref{tau-max-def} and \ref{tau-min-def}.    
  
  We know $(\gjg )\succeq \gamma_2$, so len$(\gjg) \le $ len$(\gamma_2) = j$ by Lemma~\ref{length-lemma}.
  By Lemma~\ref{tau-pr-jk-lem} and Lemma~\ref{tau-pr-max-lem}, $\tjnpr \succeq \tknpr \succeq \gamma_1$, and $\tjnpr \succeq \gamma_2$.   So $\tjnpr$ is an upper bound for  $\gamma_1$ and $\gamma_2$, and thus $\tjnpr \succeq (\gjg)$.  Therefore, 
  $j = len(\tjnpr) \le len(\gjg) \le j$, so equality holds throughout and len$(\gjg) = j = \min \{j,k\}$.
  
  The second proof is analogous. 
  \end{proof}
  
  \begin{Prop}  The poset $Dis_k(n)$ is a lattice, and a sublattice of $Dis(n)$. 
  \label{dis-k-n-lattice-prop}
  \end{Prop}
  
  \begin{proof}
  By definition, 
  $Dis_k(n) $ is a subposet of $Dis(n)$. Since $Dis(n)$ is a lattice, any two elements have both a meet and a join. By Proposition~\ref{mj-length-prop}, for any two elements  in $Dis_k(n)$, their  meet and join  are in $Dis_k(n)$. 
  \end{proof}
   
   We end this section with a remark  that follows from the identity   $\binom{k+2}{2} - 2 = \binom{k+1}{2} + k - 1$.  This remark will be useful in describing partitions of the $S$-block lattice in Section~\ref{sec-4}.
 \begin{remark}
 Let $n, k$ be integers so that $n \ge 2$, $k \ge 2$ and $n \ge \binom{k+1}{2}$.
 
{\rm  (i)}  If $n \le \binom{k+2}{2} - 2$ then all elements of $Dis_k(n)$ have last part equal to 1.

{\rm  (ii) }  If $n \ge \binom{k+2}{2} - 1$, then $Dis_k(n)$ contains a partition with last part equal to 2.
 
 \label{end-in-one-rem}
 \end{remark}
 
   For example,  we see in Figure~\ref{fig-dis-10-12}  that all elements of $Dis_4(12)$ have last part equal to 1 and there exist elements of $Dis_3(12)$ with last part equal to 2.
 
 

  \subsection{The poset $S$-$Block(n)$}
  \label{block-sec}
    
 When $n_1$ and $n_2$ are positive integers of the same parity with $n_2 \ge n_1$, 
 we combine elements of $Dis(n_1)$ and $Dis(n_2)$ into ordered pairs that we call \emph{blocks} as follows.

 \begin{Def} {\rm    Let $n_1,n_2$ be positive integers of the same parity and let $\alpha \in Dis(n_1)$ and $\beta \in Dis(n_2)$.    The  ordered pair $[\alpha|\beta]$  forms a  \emph{block}    if  the following three conditions hold.
   
 (i)  $n_2 \ge n_1$
 
 (ii)  $\beta \succeq_w \alpha$
 
 (iii)  len$(\alpha) = $ len$(\beta)$ or len$(\alpha) = $ len$(\beta) + 1.$ 
 
 \smallskip
 
 \noindent
If  in addition, $n_1 = n_2$, then $[\alpha|\beta]$ is a \emph{split-block} (or  \emph{S-block})   and in this case, (ii) is equivalent to  $\beta \succeq \alpha$.

 }
 \label{block-defn}
 \end{Def}

 In addition to the majorization that occurs within a block, there is also  majorization between blocks.

  \begin{Def} {\rm   If  $\alpha_1, \alpha_2 \in Dis(n_1)$ and  $ \beta_1, \beta_2 \in Dis(n_2)$, then 
 $[\alpha_1|\beta_1]$  majorizes $[\alpha_2|\beta_2]$  (denoted $[\alpha_1|\beta_1] \succeq [\alpha_2|\beta_2]$) if $\alpha_1 \succeq \alpha_2$ and 
 $\beta_1 \preceq \beta_2$.
 }
 \label{block-maj-def}
 \end{Def}

  In Section~\ref{sec-6}
  we consider blocks for which $n_1 \neq n_2$;  until then, we study $S$-blocks.
It is not hard to check that  
 for a fixed $n$, the set of $S$-blocks forms a poset under the majorization in Definition~\ref{block-maj-def} (with $n_1=n_2$), and we denote this poset by $S$-$Block(n)$.   The poset $S$-$Block(10)$ is shown in Figure~\ref{fig-s-10}  and is constructed from pairs of elements of  $Dis(10)$.  As shown in Figure~\ref{fig-dis-10-12}(a),  $Dis(10)$ is ranked and has left-right symmetry, and as a result, $S$-$Block(10)$ is also ranked and has left-right symmetry.    However, in general, $Dis(n)$  has neither of these properties.  Figure~\ref{fig-dis-10-12}(b) shows that $Dis(12)$ does not have left-right symmetry and Figure~\ref{fig-dis-10-12}(c) shows that $Dis(14)$  is not ranked.  As a consequence,  $S$-$Block(12)$ does not have left-right symmetry and $S$-$Block(14)$  is not ranked.

   Merris \cite{Me03} studies  a larger  class of posets and
 translates classical characterization results for split graphs and threshold graphs into conditions about split blocks, as we record in Theorems~\ref{split-S-block-thm} and \ref{thresh-S-block-thm}.  
  We provide a short  proof of Theorem~\ref{split-S-block-thm}  that translates between our    terminology and that in \cite{Me03}.

 \begin{Thm}
 The  partition $\pi$ is the degree sequence of a split graph if and only if  $[\alpha(\pi)|\beta(\pi)]$ is an $S$-block.
 \label{split-S-block-thm}
   \end{Thm}

 \begin{proof}
 If $\pi$ is the degree sequence of a split graph, then $\beta(\pi) \succeq \alpha(\pi)$ (see for example, Theorem~4.4 of 
 \cite{CoTr21}).   The first two conditions of Definition~\ref{block-defn} follow immediately, and the third condition follows from Theorem~\ref{length-alpha}.
 
 Conversely, suppose $[\alpha(\pi)|\beta(\pi)]$ is an $S$-block.  By Definition~\ref{block-defn}, we have $\beta(\pi) \succeq \alpha(\pi)$  and $\alpha(\pi), \beta(\pi) \in Dis(n)$ for some integer $n$, thus $\pi$ partitions $2n$.   Since $\beta(\pi) \succeq_w \alpha(\pi)$ and $\pi$ partitions an even integer, we know $\pi$ is graphic (see for example, Theorem~4.2 of \cite{CoTr21}) and then since $\beta(\pi) \succeq \alpha(\pi)$ we conclude that $\pi$ is the degree sequence of a split graph (again, see Theorem~4.4 of 
 \cite{CoTr21}). 
 \end{proof}

  Theorem~\ref{split-S-block-thm} provides our motivation for the name $S$-block and is possible because membership in the class of split graphs is determined by degree sequence.      Similarly, membership in the classes {balanced} split graphs, unbalanced split graphs, threshold graphs, NG-1 graphs, NG-2 graphs, NG-3 graphs, and pseudo-split graphs can also be determined by degree sequences.   The characterization theorem for threshold graphs from \cite{Me03}
   is recorded in Theorem~\ref{thresh-S-block-thm} below and the analogous results for the remaining classes are proven in Theorems~\ref{unbal-block-thm}, \ref{NG-block-thm}, and  \ref{NG3-thm} and Corollary~\ref{pseudo-charac}.
 Merris and Roby \cite{MeRo05} focus on  posets whose elements are $\alpha(\pi)$ where $\pi$ is the degree sequence of a threshold graph;  they refer to  $A(\pi)$ as a \emph{shifted shape}. 

 \begin{Thm}
 The partition $\pi$ is the degree sequence of a  threshold graph if and only if  $[\alpha(\pi)|\beta(\pi)]$ is an $S$-block and $\alpha(\pi) = \beta(\pi)$.
 \label{thresh-S-block-thm}
   \end{Thm}

  We refer to an $S$-block $\ab$     as \emph{balanced} if there is a balanced split graph that has degree sequence $\pi$ with $\alpha = \alpha(\pi)$ and $\beta = \beta(\pi)$.  Analogously, we refer to $S$-blocks as being unbalanced, threshold, NG-1, NG-2, NG-3, or pseudo-split.

   We conclude this section by showing that the set of  balanced $S$-blocks is downward closed (i.e., any S-block below a balanced S-block is balanced) and  the set of unbalanced $S$-blocks is upward-closed, i.e.,  any $S$-block above an unbalanced $S$-block is unbalanced.

\begin{Thm} Let $\abone, \abtwo  \in S$-$Block(n)$ with  $\abone \succeq \abtwo$.  If   $\abone$ is balanced, then $\abtwo$ is balanced.  If $\abtwo$ is unbalanced, then $\abone$ is unbalanced.
\label{below-balanced-prop}
\end{Thm}

\begin{proof}
Using Lemma~\ref{length-lemma}
we know len$(\alpha_1) \le $ len$(\alpha_2)$ and len$(\beta_1) \ge $ len$(\beta_2)$.   
In the case that  $\abone$ is balanced,   Theorem~\ref{unbal-block-thm} implies that 
 len$(\alpha_1) = $ len$(\beta_1) + 1$.     Thus
len$(\alpha_2) \ge $ len$(\alpha_1) = $ len$(\beta_1) + 1 \ge $ len$(\beta_2) + 1$.  Since $\abtwo$ is an $S$-block, by Definition~\ref{block-defn}
we  conclude len$(\alpha_2)  = $ len$(\beta_2) + 1$ and thus $\abtwo $ is balanced by  
 Theorem~\ref{unbal-block-thm}.  The last sentence  then follows  as the contrapositive. 
\end{proof}

  \section{Partitioning $S$-$Block(n)$ into Amphoras}
   \label{sec-4}

In this section, we study the structure of $S$-$Block(n)$ and provide our main theorems which link subposets called amphoras to sets of $S$-blocks representing subclasses of split graphs.
Amphoras are ceramic pots used in ancient Greece to transport oils and other liquids in shipping.  They have a pointed bottom and are wider in the middle and at the top, but do not always have a cone shape.    Their shape allowed for them to be packed in the hold of ships in two layers so that there was minimal movement during travel.  In a similar manner, our amphoras will each have a unique minimum element, an antichain of maximal elements, and contain everything between the minimum element and one or more of the maximal elements. Figure~\ref{fig-s-10}  shows $S$-$Block(10)$ partitioned into amphoras outlined in solid and dotted lines.  
  
 \begin{Def} {\rm 
 Given a poset $P$, a  subposet $P'$   forms an \emph{amphora} if  there exists an antichain $A$ in $P$ and an element $x$ less than  or equal to every element of $A$ so that $P'$ consists of all elements of $P$ that are greater than or equal to $x$ and less than or equal to some element of $A$. 
}
\label{amphora-def}
 \end{Def}

 Equivalently, we can think of an amphora as a union of intervals in a poset where the bottom element  of each interval is the same  and the top elements are the members of an antichain.  
We are interested in  amphoras that partition  the poset $S$-$Block(n)$ 
according to length. 
 For each $k$ (with $n \ge \binom{k+1}{2}$)  we will show that there is a unique minimal unbalanced $S$-block  $\ab$ where len$(\alpha) = $  len$(\beta) = k$.  Similarly, there is a unique minimal balanced $S$-block  $\ab$ where len$(\alpha) =k$ and  len$(\beta) = k-1$.
These minimal elements will not be \emph{minimum} in the full poset $S$-$Block(n)$, but will be within certain amphoras.
    Our partition will result in two layers, with balanced $S$-blocks on the bottom and unbalanced $S$-blocks above.   In Figure~\ref{fig-s-10}, the amphoras for balanced $S$-blocks are outlined in blue dotted lines and those for unbalanced $S$-blocks are outlined in solid yellow lines.  We define  the  sets  $A(n,k)$   and  $A(n,k,k-1)$ in  Definitions~\ref{Ak-def}  and \ref{Akk-def} and show that they are amphoras and that they   partition the unbalanced $S$-blocks and the balanced $S$-blocks respectively.

\subsection{The Amphoras $A(n,k)$ and $A(n,k,k-1)$}

The   maximal elements of $S$-$Block(n)$ are characterized in Lemma~\ref{max-elts-lem} and correspond to threshold graphs by Theorem~\ref{thresh-S-block-thm}.

\begin{lemma}
The  set of maximal elements of $S$-$Block(n)$    is $\{[\gamma|\gamma]: \gamma \in Dis(n)\}.$
\label{max-elts-lem}
\end{lemma}

\begin{proof}
 Suppose $ \gamma \in Dis(n)$ and $\ab \in S$-$Block(n)$ with  $\ab \succeq [\gamma|\gamma]$.  By Definition~\ref{block-defn}, $\beta \succeq \alpha$, and by  Definition~\ref{block-maj-def}, $\alpha \succeq \gamma \succeq \beta$, thus $\alpha = \beta  = \gamma$ and  $[\gamma|\gamma]$ is maximal   in $S$-$Block(n)$.  For any  $\abpr \in S$-$Block(n)$,  Definition~\ref{block-defn} implies that  $\abpr \preceq [\alpha'|\alpha'] $, thus the only maximal elements of $S$-$Block(n)$   are those of the form $[\gamma|\gamma]$.
\end{proof}

 We partition the  antichain of maximal elements of $S$-$Block(n)$ according to length as follows.

\begin{Def}
$(Max)^n_{k} = \{ [\gamma|\gamma] \in S$-Block(n):  $\gamma \in Dis_k(n) \}$

\end{Def}

   In the next definition, we give a  partition of the   unbalanced  elements of $S$-$Block(n)$, and Theorem~\ref{Ak-thm}  shows that these sets are amphoras and characterizes their minimum and maximal elements.

\begin{Def} {\rm
For $n \ge \binom{k+1}{2}$, define  $A(n,k)$ to be the subposet of  $S$-$Block(n)$ consisting of all  (unbalanced) $S$-blocks $\ab$ where $\alpha, \beta \in Dis_k(n)$. }
\label{Ak-def}
\end{Def}

\noindent For example, $A(10,3)$ is outlined in solid yellow in Figure~\ref{fig-s-10} and has maximal elements $[5,3,2 \, | \, 5,3,2],$ $ [5,4,1 \, | \, 5,4,1], \  [6,3,1 \, | \, 6,3,1],  \ [7,2,1 \, | \, 7,2,1]$. It has minimum element $[5,3,2 \, | \, 7,2,1]$.   Note that  $[5,3,2 \, | \, 7,2,1]$ is minimum in $A(10,3)$ but not in $S$-$Block(10)$.  More generally,   $A(n,k)$ has minimum element $\tkntknpr$, where $\tkn$ and $\tknpr$ are defined in Definitions~\ref{tau-max-def} and \ref{tau-min-def}.

\begin{Thm}
The set $A(n,k)$  of $S$-Block(n)  is an amphora  with minimum element
$\tkntknpr$
and whose antichain of maximal elements is  $(Max)^n_{k}.$
\label{Ak-thm}
\end{Thm}

\begin{proof}
First consider $\ab \in A(n,k)$, so $\beta \succeq \alpha$ and $\alpha, \beta \in Dis_k(n)$.  By Definition~\ref{block-maj-def}, $\ab$ is majorized by the elements 
 $[\alpha|\alpha]$ and $[\beta|\beta]$ of $(Max)^n_{k}$, and by Lemmas~\ref{tau-pr-max-lem}, \ref{tau-min-lem}, and Definition~\ref{block-maj-def}, $\ab$ majorizes $\tkntknpr$.
 
 Conversely, let $\ab$ be an $S$-block that majorizes $\tkntknpr$ and is majorized by an element of $(Max)^n_{k}$.  Since $\tkn$, $\tknpr$, and all elements of $(Max)^n_{k}$ are in $Dis_k(n)$, Lemma~\ref{length-lemma} implies that $\alpha,\beta \in Dis_k(n)$ and thus $\ab \in A(n,k)$.
 \end{proof}

 As a consequence of the proof of Theorem~\ref{Ak-thm}, we can specify   the set of all minimal unbalanced split blocks as follows.
 
\begin{Cor}
In   $S$-$Block(n)$,  the  minimal unbalanced split blocks are the elements 
$  \tkntknpr $ for which   $n \ge \binom{k+1}{2}.$

\end{Cor}

 In order to describe the antichain of maximal elements of the amphoras for 
  balanced $S$-blocks we refer to the covering relation in the posets $Dis(n)$ and $S$-$Block(n)$.

 \small 
\begin{figure}[h]
\begin{center}
\begin{tikzpicture}[scale=.3]

\draw[yellow, ultra thick,fill=yellow!5] (-3.5,100.5) -- (-1.2,98.5) --(-1,98.5)-- (1.5,100.5) -- (-1.2,101.5)--(-1.5,101.5)--cycle;: 
\draw[yellow, ultra thick,fill=yellow!5] (2,100.5) -- (10,88) --(10.2,88)-- (22,100.5) -- (14,102.5)--(9,102.5)--cycle;: 
\draw[yellow, ultra thick,fill=yellow!5] (24,100.5) -- (33,89) --(33.2,89)-- (42,100.5) -- (35,102.5)--(30,102.5)--cycle;: 
\draw[yellow, ultra thick,fill=yellow!5] (43,100.5) -- (45,98) --(45.2,98)-- (47,100.5) -- (45.6,101.5)--(45.5,101.5)--cycle;: 


\draw[blue, dotted, ultra thick,fill=blue!3] (-4,97) -- (2.5,87)--(8,85) --(8.6,85)-- (8.3,90)--(2.7,98.5) -- (0,98.5)--cycle;: 
\draw[blue, dotted, ultra thick,fill=blue!3] (13.5,91) -- (21,78) --(21.2,78)-- (31,91) -- (25.3,98)--(20,98)--cycle;: 
\draw[blue, dotted, ultra thick,fill=blue!3] (40,97) -- (35,90)--(33,85) --(33.8,85)-- (42,90)--(46,97)--(44,98.5) -- (43.5,98.5)--cycle;:

\tiny \color{red} 
\node at (-1,100) {
${\tiny   \color{red}\left[\begin{array}{c}
\color{black} \mbox{4,3,2,1} \\
\color{cyan}\mbox{4,3,2,1}
\end{array}\right]}$
};
\node at (5,100) {
${\tiny \left[\begin{array}{c}
\color{purple}\mbox{5,3,2} \\
\color{black} \mbox{5,3,2}
\end{array}\right]}$
};
\node at (10,100) {
${\tiny  \left[\begin{array}{c}
\color{black} \mbox{5,4,1} \\
\color{cyan}\mbox{5,4,1}
\end{array}\right]}$
};
\node at (15,100) {
${\tiny \left[\begin{array}{c}
\color{black} \mbox{6,3,1} \\
\color{cyan}\mbox{6,3,1}
\end{array}\right]}$
};

\node at (19,100) {
${\tiny  \left[\begin{array}{c}
\color{black} \mbox{7,2,1} \\
\color{cyan}\mbox{7,2,1}
\end{array}\right]}$
};
\node at (26,100) {
${\tiny \left[\begin{array}{c}
\color{purple}\mbox{6,4} \\
\color{black} \mbox{6,4}
\end{array}\right]}$
};
\node at (30,100) {
${\tiny  \left[\begin{array}{c}
\color{purple}\mbox{7,3} \\
\color{black} \mbox{7,3}
\end{array}\right]}$
};
\node at (35,100) {
${\tiny  \left[\begin{array}{c}
\color{purple}\mbox{8,2} \\
\color{black} \mbox{8,2}
\end{array}\right]}$
};
\node at (40,100) {
${\tiny \left[\begin{array}{c}
\color{black} \mbox{9,1} \\
\color{cyan}\mbox{9,1}
\end{array}\right]}$
};
\node at (45,100) {
${\tiny \left[\begin{array}{c}
\color{purple}\mbox{10} \\
\color{black} \mbox{10}
\end{array}\right]}$
};

\color{black} 
\node at (0,97) {
${\tiny \left[\begin{array}{c}
\mbox{4,3,2,1} \\
\mbox{5,3,2}
\end{array}\right]}$
};
\node at (6,97) {
${\tiny \left[\begin{array}{c}
\color{purple}\mbox{5,3,2} \\
\color{cyan}\mbox{5,4,1}
\end{array}\right]}$
};
\node at (11,97) {
${\tiny \left[\begin{array}{c}
\mbox{5,4,1} \\
\color{cyan}\mbox{6,3,1}
\end{array}\right]}$
};
\node at (16,97) {
${\tiny \left[\begin{array}{c}
\mbox{6,3,1} \\
\color{cyan}\mbox{7,2,1}
\end{array}\right]}$

};\node at (21,97) {
${\tiny \left[\begin{array}{c}
\mbox{6,3,1} \\
\mbox{6,4}
\end{array}\right]}$
};
\node at (24.5,97) {
${\tiny \left[\begin{array}{c}
\mbox{7,2,1} \\
\mbox{7,3}
\end{array}\right]}$
};
\node at (29,97) {
${\tiny \left[\begin{array}{c}
\color{purple}\mbox{6,4} \\
\mbox{7,3}
\end{array}\right]}$
};
\node at (34,97) {
${\tiny \left[\begin{array}{c}
\color{purple}\mbox{7,3} \\
\mbox{8,2}
\end{array}\right]}$
};\node at (37.5,97) {
${\tiny \left[\begin{array}{c}
\color{purple}\mbox{8,2} \\
\color{cyan}\mbox{9,1}
\end{array}\right]}$
};
\node at (43.5,97) {
${\tiny \left[\begin{array}{c}
\mbox{9,1} \\
\mbox{10}
\end{array}\right]}$
};


\node at (1,94) {
${\tiny \left[\begin{array}{c}
\mbox{4,3,2,1} \\
\mbox{5,4,1}
\end{array}\right]}$
};
\node at (8,94) {
${\tiny \left[\begin{array}{c}
\color{purple}\mbox{5,3,2} \\
\color{cyan}\mbox{6,3,1}
\end{array}\right]}$
};
\node at (13,94) {
${\tiny \left[\begin{array}{c}
\mbox{5,4,1} \\
\color{cyan}\mbox{7,2,1}
\end{array}\right]}$
};
\node at (18.5,94) {
${\tiny \left[\begin{array}{c}
\mbox{5,4,1} \\
\mbox{6,4}
\end{array}\right]}$
};
\node at (22.2,94) {
${\tiny \left[\begin{array}{c}
\mbox{6,3,1} \\
\mbox{7,3}
\end{array}\right]}$
};

\node at (26,94) {
${\tiny \left[\begin{array}{c}
\mbox{7,2,1} \\
\mbox{8,2}
\end{array}\right]}$
};
\node at (31.5,94) {
${\tiny \left[\begin{array}{c}
\color{purple}\mbox{6,4} \\
\mbox{8,2}
\end{array}\right]}$
};
\node at (35,94) {
${\tiny \left[\begin{array}{c}
\color{purple}\mbox{7,3} \\
\color{cyan} \mbox{9,1}
\end{array}\right]}$
};
\node at (41,94) {
${\tiny \left[\begin{array}{c}
\mbox{8,2} \\
\mbox{10}
\end{array}\right]}$
};


\node at (4,91) {
${\tiny \left[\begin{array}{c}
\mbox{4,3,2,1} \\
\mbox{6,3,1}
\end{array}\right]}$
};
\node at (10.5,91) {
${\tiny \left[\begin{array}{c}
\color{purple} \mbox{5,3,2} \\
\color{cyan}\mbox{7,2,1}
\end{array}\right]}$
};
\node at (16,91) {
${\tiny \left[\begin{array}{c}
\mbox{5,3,2} \\
\mbox{6,4}
\end{array}\right]}$
};
\node at (20.4,91) {
${\tiny \left[\begin{array}{c}
\mbox{5,4,1} \\
\mbox{7,3}
\end{array}\right]}$
};

\node at (25,91) {
${\tiny \left[\begin{array}{c}
\mbox{6,3,1} \\
\mbox{8,2}
\end{array}\right]}$
};
\node at (28.8,91) {
${\tiny \left[\begin{array}{c}
\mbox{7,2,1} \\
\mbox{9,1}
\end{array}\right]}$
};
\node at (33,91) {
${\tiny \left[\begin{array}{c}
\color{purple}\mbox{6,4} \\
\color{cyan} \mbox{9,1}
\end{array}\right]}$
};
\node at (39,91) {
${\tiny \left[\begin{array}{c}
\mbox{7,3} \\
\mbox{10}
\end{array}\right]}$
};

\node at (6,88) {
${\tiny \left[\begin{array}{c}
\mbox{4,3,2,1} \\
\mbox{7,2,1}
\end{array}\right]}$
};

\node at (17.3,88) {
${\tiny \left[\begin{array}{c}
\mbox{5,3,2} \\
\mbox{7,3}
\end{array}\right]}$
};
\node at (22,88) {
${\tiny \left[\begin{array}{c}
\mbox{5,4,1} \\
\mbox{8,2}
\end{array}\right]}$
};

\node at (26.7,88) {
${\tiny \left[\begin{array}{c}
\mbox{6,3,1} \\
\mbox{9,1}
\end{array}\right]}$
};

\node at (36,88) {
${\tiny \left[\begin{array}{c}
\mbox{6,4} \\
\mbox{10}
\end{array}\right]}$
};

\node at (19,85) {
${\tiny \left[\begin{array}{c}
\mbox{5,3,2} \\
\mbox{8,2}
\end{array}\right]}$
};

\node at (24.5,85) {
${\tiny \left[\begin{array}{c}
\mbox{5,4,1} \\
\mbox{9,1}
\end{array}\right]}$
};

\node at (22,82) {
${\tiny \left[\begin{array}{c}
\mbox{5,3,2} \\
\mbox{9,1}
\end{array}\right]}$
};

{ \color{teal}
\draw[thick] (-1.3,99)--(0.5,98);
\draw[thick] (5,99)--(6,98);
\draw[thick] (10,99)--(11,98);
\draw[thick] (15,99)--(16,98);
\draw[thick] (15.5,99)--(20.5,98.2);
\draw[thick] (19.3,99)--(23.8,98.3);
\draw[thick] (26.5,99)--(28.5,98);
\draw[thick] (30.5,99)--(33.5,98);
\draw[thick] (35.2,99)--(37.5,98);
\draw[thick] (40,99)--(43.3,98);
\draw[thick] (0,96)--(1,95);
\draw[thick] (6,96)--(7.5,95);
\draw[thick] (10.5,96)--(12.3,95);
\draw[thick] (11.5,96)--(18.2,95);
\draw[thick] (16.5,96)--(21.2,95);
\draw[thick] (21,96)--(21.7,95);
\draw[thick] (24.5,96)--(25.5,95);
\draw[thick] (29,96)--(31.3,95);
\draw[thick] (34.3,96)--(34.8,95);
\draw[thick] (38,96)--(40.5,95);
\draw[thick] (2,93)--(3,92);
\draw[thick] (8,93)--(10,92);
\draw[thick] (8.6,93)--(15.5,92);
\draw[thick] (13,93)--(19.5,92);
\draw[thick] (18.3,93)--(20.5,92);
\draw[thick] (22,93)--(24.5,92);
\draw[thick] (26.2,93)--(28.5,92);
\draw[thick] (31.5,93)--(33,92);
\draw[thick] (35.5,93)--(39,92);

\draw[thick] (4.2,90)--(5.5,89);
\draw[thick] (11.2,90)--(16.3,89);
\draw[thick] (16,90)--(17,89);
\draw[thick] (20.8,90)--(21.8,89);
\draw[thick] (25.7,90)--(26.7,89);
\draw[thick] (33,90)--(36,89);

\draw[thick] (17,87)--(18.5,86);
\draw[thick] (22,87)--(24.5,86);
\draw[thick] (19,84)--(21.5,83);

}

{ \color{violet}
\draw[thick] (4.5,99)--(1,98);
\draw[thick] (9.5,99)--(6.5,98);
\draw[thick] (14.5,99)--(11.5,98);
\draw[thick] (19,99)--(16.5,98);
\draw[thick] (25.3,99)--(21.3,98.2);
\draw[thick] (29.5,99)--(24.8,98.2);
\draw[thick] (30,99)--(29,98);
\draw[thick] (35,99)--(34,98);
\draw[thick] (39.5,99)--(38,98);
\draw[thick] (44.5,99)--(44,98);

\draw[thick] (5.4,96)--(2.2,95);
\draw[thick] (10,96)--(8,95);
\draw[thick] (15.3,96)--(13,95);
\draw[thick] (20,96)--(18.5,95);
\draw[thick] (24,96)--(22,95);
\draw[thick] (28.5,96)--(23,95);
\draw[thick] (33.2,96)--(26.2,95);
\draw[thick] (33.9,96)--(32,95);
\draw[thick] (37.2,96)--(35.5,95);
\draw[thick] (43,96)--(41.2,95);
\draw[thick] (7.5,93)--(5,92);
\draw[thick] (12.5,93)--(11,92);
\draw[thick] (17.8,93)--(16,92);
\draw[thick] (21.5,93)--(20.7,92);
\draw[thick] (25.7,93)--(25,92);
\draw[thick] (31,93)--(25.5,92);
\draw[thick] (34.5,93)--(29,92);
\draw[thick] (35,93)--(33.4,92);
\draw[thick] (41,93)--(39.3,92);
\draw[thick] (10.5,90)--(6,89);
\draw[thick] (20,90)--(17.3,89);
\draw[thick] (25,90)--(22.5,89);
\draw[thick] (28.2,90)--(27.2,89);
\draw[thick] (32.5,90)--(27.7,89);
\draw[thick] (38.8,90)--(36.5,89);

\draw[thick] (21.5,87)--(19.4,86);
\draw[thick] (27,87)--(25.5,86);
\draw[thick] (25,84)--(22.5,83);
}

\end{tikzpicture}

\caption{The poset $S$-$Block(10)$ partitioned into amphoras. Unbalanced amphoras are outlined in a solid line and balanced amphoras are outlined in a dashed line.}
\label{fig-s-10}
\end{center}
\end{figure}
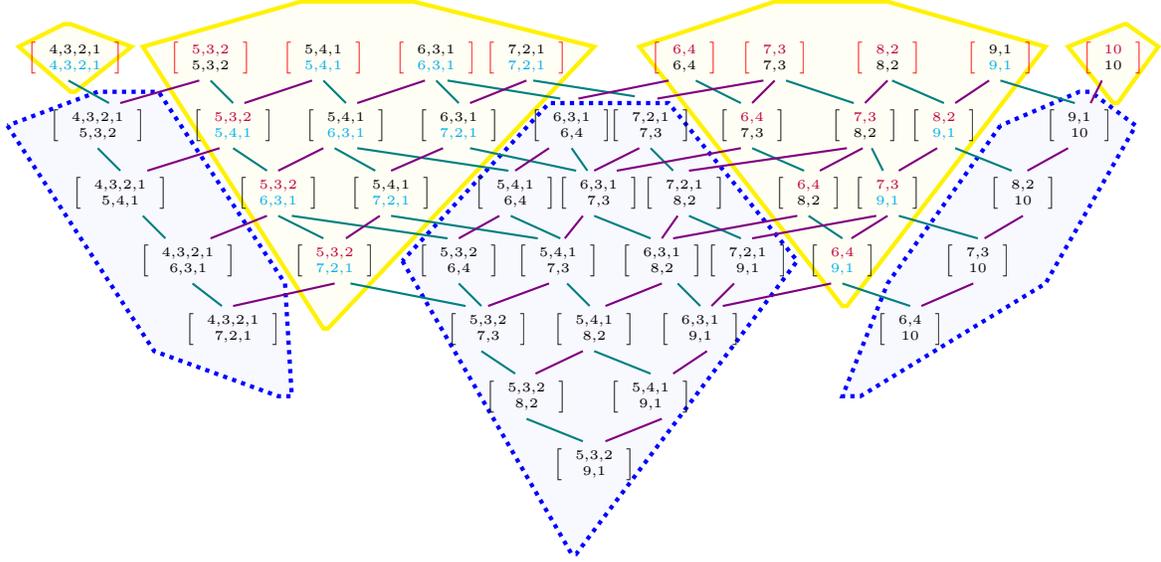

We denote  by $\gtrdot$ the  covering relation    in $Dis(n)$ and  $S$-$Block(n)$.
Observe that in Figure~\ref{fig-s-10}, when one block covers another in $S$-$Block(10)$, either the first components of the blocks are equal and the second components are a covering relation in $Dis(n)$, or vice versa.  We make this more precise in Proposition~\ref{covering-prop}.

  \begin{Prop}
  Let $\abone, \abtwo \in S$-$Block(n)$ for some integer $n$.  If $\abone \gtrdot \abtwo$ then either  (i) $\alpha_1 = \alpha_2$  and $\beta_1 \lessdot \beta_2$ or (ii) $\beta_1 = \beta_2$  and $\alpha_1 \gtrdot \alpha_2$.
  \label{covering-prop}
  \end{Prop}
   
  \begin{proof}
  If $\alpha_1 \neq \alpha_2$ and $\beta_1 \neq \beta_2$, then $\abone \succ [\alpha_1|\beta_2] \succ \abtwo$, contradicting the hypothesis.  Thus $\alpha_1 = \alpha_2$ or $\beta_1 = \beta_2$.  In the first case, if there exists $\beta \in Dis(n)$  that is distinct from $\beta_1$ and $\beta_2$ and satisfies  $\beta_1 \prec \beta \prec \beta_2$, then $\abone \succ [\alpha_1|\beta] \succ \abtwo$, again contradicting the hypothesis.  Thus  $\alpha_1 = \alpha_2$ implies $\beta_1 \lessdot \beta_2$.
   The second case is analogous.
    \end{proof}

  We next wish to   partition  the   balanced  elements of $S$-$Block(n)$ into amphoras, and to do so, we need to identify the  elements that will be maximal in these amphoras.  These arise from the set of threshold-covered $S$-blocks, studied in \cite{Me03}.
  
\begin{Def}
{\rm The set of \emph{threshold-covered}  members of $S$-$Block(n)$  is\\
$TC(n) = \{ \ab \in S$-$Block(n): \beta \gtrdot \alpha\}$. }
\end{Def}

The name \emph{threshold-covered} comes from that fact that the elements of $TC(n)$ are precisely the members of $S$-$Block(n)$ that are covered by an $S$-block of a threshold graph.  This follows as a consequence of the proof of the next result.

\begin{Prop}
 $TC(n)$ is an antichain.
\label{TC-prop}
\end{Prop}

\begin{proof}
Let $\ab \in TC(n)$, thus $\beta \gtrdot \alpha$.   We will show that distinct elements of $TC(n)$ are incomparable by showing that there are precisely two elements of $S$-$Block(n)$ that are greater than $\ab$, namely  $[\alpha|\alpha]$ and $[\beta|\beta]$.  
Suppose $\abpr \in S$-$Block(n)$ and $\ab \preceq \abpr$.  By the definition of majorization and   $\abpr$ being an $S$-block, we have 
$\alpha \preceq \alpha' \preceq \beta' \preceq \beta$.  If two of these inequalities are strict, we contradict $\beta \gtrdot \alpha$.  Thus at most one of the inequalities is strict and we get $\abpr$ equal to either $\ab$, $[\alpha|\alpha]$ or  $[\beta|\beta]$.  Thus the only $S$-blocks greater than $\ab$ are $[\alpha|\alpha]$ or  $[\beta|\beta]$.   The result follows immediately since elements of $TC(n)$ do not have the form $[\gamma|\gamma]$.
\end{proof}

We partition the subset of  balanced $S$-blocks in    $TC(n)$ according to lengths as follows.
  
\begin{Def}
{\rm
$(Max)^n_{k,k-1} = \{ \ab \in TC(n):   \alpha \in Dis_k(n), \beta \in Dis_{k-1}(n) \}$.
}
\label{max-bal}
\end{Def}

  The next definition   partitions  the balanced elements of $S$-$Block(n)$, and Theorem~\ref{Akk-thm}  shows  that each part is an amphora and characterizes the minimum and maximal elements.
  
  \begin{Def}
  {\rm
For $n \ge \binom{k+1}{2}$, define  $A(n,k,k-1)$ to be the  subposet of $S$-$Block(n)$ consisting of the  (balanced) $S$-blocks $\ab$ where $\alpha \in Dis_k(n)$ and $\beta \in Dis_{k-1}(n)$.
}
\label{Akk-def}
\end{Def}

 For example, $A(10,3,2)$ is outlined in dashed blue lines in Figure~\ref{fig-s-10} and has maximal elements $[6,3,1 \, | \, 6,4]$ and $ [7,2,1 \, | \, 7,3]$ and minimum element $[5,3,2 \; | \, 9,1]$.  The next proposition specifies these values more generally.  

\begin{Thm}
The set $A(n,k,k-1)$ is an amphora with minimum element  $[\tkn|\tau'_{k-1}(n)]$
and whose antichain of maximal elements is   $(Max)^n_{k,k-1}.$
\label{Akk-thm}
\end{Thm}

\begin{proof}
First consider $\ab \in A(n,k,k-1)$, so $\beta \succeq \alpha$ and $\alpha \in Dis_k(n)$ and  $\beta \in Dis_{k-1}(n)$.  By Lemma~\ref{tau-pr-max-lem}, $\beta \preceq \tau'_{k-1}(n)$, and by Lemma~\ref{tau-min-lem},
$\alpha \succeq \tkn$, thus $\ab \succeq [\tkn|\tau'_{k-1}(n)]$ by Definition~\ref{block-maj-def}.

We next show that $\ab$ is majorized by an element of $(Max)^n_{k,k-1}.$  In $Dis(n)$ there exists a chain of covering relations:  
$\beta = \gamma_1 \gtrdot \gamma_2 \gtrdot \gamma_3 \gtrdot  \cdots \gtrdot \gamma_m = \alpha$.  Since $\beta$ has length $k-1$ and $\alpha$ has length $k$,  by Lemma~\ref{length-lemma} there exists $r:1 \le r < m$ for which $\gamma_i \in Dis_{k-1}(n)$ for $i \le r$ and $\gamma_i \in Dis_{k}(n)$ for $i > r$.  Then $[\gamma_{r+1}|\gamma_r] \in  (Max)^n_{k,k-1}$ and $\ab \preceq [\gamma_{r+1}|\gamma_r] $. 

 Conversely, let $\ab$ be an $S$-block that majorizes $[\tkn|\tau'_{k-1}(n)]$ and is majorized by an element  $[\alpha'|\beta']$ of $(Max)^n_{k,k-1}$.  Since $\alpha', \tkn \in Dis_k(n)$ and $\alpha' \succeq \alpha \succeq \tkn$ we know $\alpha \in Dis_k(n)$ by Lemma~\ref{length-lemma}, and similarly $\beta \in Dis_{k-1}(n)$.  So $\ab \in A(n,k,k-1)$.
 \end{proof}

 A consequence of the proof of Theorem~\ref{Akk-thm} is  that we can specify the set of all minimal balanced split blocks.
 
\begin{Cor}
In   $S$-$Block(n)$,  the set of minimal balanced split blocks is
$\{ [\tau_k(n)|\tau'_{k-1}(n)]:  n \ge \binom{k+1}{2} \}.$
\end{Cor}

We now arrive at one of our main results.
  Combining the results in  Theorems~\ref{Ak-thm} 
  and \ref{Akk-thm} we conclude that  $S$-$Block(n)$ is partitioned into amphoras as we state below.
  
  \begin{Thm}
  The poset $S$-$Block(n)$ is the disjoint union of the amphoras $A(n,k)$ and $A(n,k,k-1)$ for all  $k\ge1$ satisfying $\binom{k+1}{2} \le n$.
  \label{partition-thm}
  \end{Thm}
  
  \subsection{The poset $W(n)$  of amphoras}
  
In this section, we describe the interaction between the amphoras of balanced and unbalanced $S$-blocks.  For $S$-$Block(10)$, these are highlighted   in Figure~\ref{fig-s-10} and 
the overall structure is captured in Figure~\ref{fig-W}.
We form a height one poset $W(n) $ whose elements are the amphoras  $A(n,k)$, $A(n,k,k-1)$ for $n \ge \binom{k+1}{2}$.      The set of minimal elements of $W(n)$   is $ \{A(n,k,k-1):   n\ge \binom{k+1}{2} \}$, the amphoras containing the balanced $S$-blocks, and the  set of maximal elements of $W(n)$ is $ \{A(n,k):   n\ge \binom{k+1}{2} \}$, the amphoras containing the unbalanced $S$-blocks.    The only comparabilities  in $W(n)$ are $A(n,k,k-1) \prec A(n,k)$ and $A(n,k,k-1) \prec A(n,{k-1})$, giving $W(n)$ the look of a zigzag. 
 The poset $W(10)$ is shown in Figure~\ref{fig-W}.  The comparabilities in $W(n)$ are interpreted  as statements about elements of $S$-$Block(n)$ in the following theorem.  In particular, 
we can combine the  amphora $A(n,k,k-1)$  of balanced $S$-blocks with either one or both of  the two surrounding amphoras  for unbalanced $S$-blocks to get  a larger amphora.

 \begin{Thm}
 Any amphora in $W(n)$ corresponds to an amphora in $S$-$Block(n)$.
 In particular, $A(n,k) \cup A(n,k,k-1)$,  $A(n,k,k-1) \cup  A(n,k-1)$, and $A(n,k) \cup A(n,k,k-1) \cup  A(n,k-1)$  are all amphoras of $S$-$Block(n)$.
\label{big-amphora-thm}
 \end{Thm}

  \begin{proof}
 
 First we show that $[\tau_k(n)|\tau'_{k-1}(n)]$ is below every element in $A(n,k) \cup A(n,k,k-1) \cup  A(n,k-1)$.
 By  Lemma~\ref{tau-pr-jk-lem} we have $ \tau'_{k}(n) \preceq \tau'_{k-1}(n)$, and by Lemma~\ref{tau-jk-lem} we have $\tau_{k-1}(n) \succeq \tau_{k}(n)$.  Thus $[\tau_k(n)|\tau'_{k}(n)] \succeq [\tau_k(n)|\tau'_{k-1}(n)]$ and 
  $[\tau_{k-1}(n)|\tau'_{k-1}(n)] \succeq [\tau_k(n)|\tau'_{k-1}(n)]$.  However,  
  $[\tau_{k}(n)|\tau'_{k}(n)]$ is the minimum element of $A(n,k)$,    $[\tau_k(n)|\tau'_{k-1}(n)]$ is the minimum element of $A(n,k,k-1)$, and 
 $[\tau_{k-1}(n)|\tau'_{k-1}(n)] $ is the minimum element of $A(n,k-1)$ by Theorems~\ref{Ak-thm} and \ref{Akk-thm}.
Thus we have shown that $[\tau_k(n)|\tau'_{k-1}(n)] \preceq y$ for every $y \in A(n,k) \cup A(n,k,k-1) \cup  A(n,k-1)$.

 Next we consider maximal elements.
 For any $\ab \in A(n,k,k-1)$ we have $\ab \prec [\alpha|\alpha]$  where  $[\alpha|\alpha] \in A(n,k)$, and $\ab \prec [\beta|\beta]$ where $[\beta|\beta] \in A(n, k-1)$.  
 Thus $A(n,k) \cup A(n,k,k-1)$  (respectively,  $A(n,k,k-1) \cup  A(n,k-1)$)  is   an amphora   consisting of those $S$-blocks  that majorize $ [\tau_k(n)|\tau'_{k-1}(n)]$ and are majorized by  an element of $(Max)^n_k$ (respectively   $(Max)^n_{k-1}$).   Similarly, the set 
  $A(n,k) \cup A(n,k,k-1) \cup  A(n,k-1)$  is   an amphora   consisting of those $S$-blocks  that majorize $ [\tau_k(n)|\tau'_{k-1}(n)]$ and are majorized by an element of 
   $(Max)^n_{k} \cup (Max)^n_{k-1}$.
 \end{proof}
 
  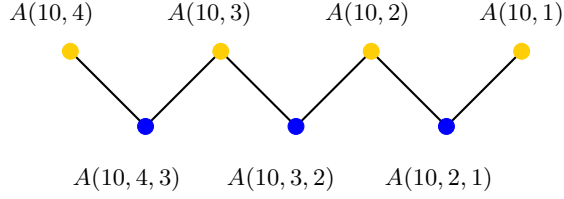
\begin{figure}[h]
  \begin{center}
\begin{tikzpicture}
\draw[thick] (0,1)--(1,0)--(2,1)--(3,0)--(4,1)--(5,0)--(6,1);
\filldraw[blue]

(1,0) circle [radius=3pt]
(3,0) circle [radius=3pt]
(5,0) circle [radius=3pt]
;

\filldraw[yellow!65!orange]

(0,1) circle [radius=3pt]
(2,1) circle [radius=3pt]
(4,1) circle [radius=3pt]
(6,1) circle [radius=3pt]
;

\node(0) at (-.25,1.5) {\small $ A(10,4)$};
\node(0) at (1.85,1.5) {\small $ A(10,3)$};
\node(0) at (3.95,1.5) {\small $ A(10,2)$};
\node(0) at (6,1.5) {\small $ A(10,1)$};

\node(0) at (.75,-.7) {\small $ A(10,4,3)$};
\node(0) at (2.8,-.7) {\small $ A(10,3,2)$};
\node(0) at (4.9,-.7) {\small $ A(10,2,1)$};
  \end{tikzpicture}
  \end{center} 

  
  \caption{The poset $W(10)$.}
  \label{fig-W}
  \end{figure}

    In Figure~\ref{fig-s-10} it is easy to see that every $S$-block in $A(10,3)$ is incomparable to every $S$-block in $A(10,2)$, and this is also seen in the incomparabilities in $W(10)$ shown in Figure~\ref{fig-W}.   A more general statement is given in Theorem~\ref{incomp-thm} below where we use the symbol $\parallel$ to denote incomparability.  The proof follows from  Lemma~\ref{length-lemma}
    and Definition~\ref{block-maj-def}.

\begin{Thm}  Suppose $\abone$ and $\abtwo$ are  $S$-blocks and $A_1, A_2$ are amphoras in $ W(n)$ with $\abone \in A_1$ and $\abtwo \in A_2$.
\medskip

{\rm (i)} 
If $A_1 \parallel A_2$ in $W(n)$ then $\abone \parallel \abtwo$ in $S$-$Block(n)$. 

{\rm (ii)}  If $A_1 \preceq A_2$ in $W(n)$ then  either $\abone \prec \abtwo$ or $\abone \parallel \abtwo$ in $S$-$Block(n)$.

{\rm (iii)} If $\abone \prec \abtwo$  then  $A_1 \preceq A_2$.

\label{incomp-thm}
\end{Thm}

 \begin{Cor}
There is a comparability $A_1 \prec A_2$ in poset $W(n)$ precisely when there is an $S$-block $\abone \in A_1$ and an $S$-block $\abtwo \in A_2$ for which $\abone \prec \abtwo$.
\label{cor-good}
 \end{Cor}
 
 \subsection{Amphoras for NG-1 and NG-2 $S$-blocks}
 
 For a split graph $G$ with degree sequence $\pi$, Theorem~\ref{unbal-block-thm} allows us to determine  whether $G$ is balanced or unbalanced directly from $\alpha(\pi)$ and $\beta(\pi)$.
 The next theorem does the same for NG-1 and NG-2 graphs.   Recall from Proposition~\ref{unbal-NG-prop} that NG-1 and NG-2 graphs are unbalanced split graphs; thus we need not consider balanced split graphs.

 \begin{Thm}
Let $[\alpha|\beta]$ be an $S$-block corresponding to the degree sequence of an unbalanced split graph $G$.  Then $G$ is  an NG-1  graph if  and only if the last part of partition $\beta$ is 1 and $G$ is an  NG-2 graph if and only if the last part of partition $\alpha$ is at least 2.
\label{NG-block-thm}
\end{Thm}

\begin{proof}  Let the degree sequence of $G$ be $\pi: d_1 \ge d_2 \ge \cdots \ge  d_k$, and let $\alpha = \alpha(\pi)$ and $\beta = \beta(\pi)$.   Furthermore,
let $m = \max \{i: d_i \ge i-1\}$,  thus $d_m \ge m-1$ and $d_{m+1} < m$. By Theorem~\ref{length-alpha}, 
  len$(\beta) = m-1$, so there is a box in $F(\pi)$ in row $m$, column $m-1$, and hence also above it in row $m-1$, column $m-1$.  
  Since $G$ is unbalanced, Theorem~\ref{unbal-block-thm} implies that     len$(\alpha) = m-1$, thus there is no box in row $m$, column $m$ in $F(\pi)$. 
  Now the definition of $\alpha$ implies that $d_m = m-1$.  
    The last part of  $\beta$ equals 1 precisely when there is no box in $F(\pi)$ in row $m+1$, column $m-1$.  This occurs  if and only if $m$ is the largest index $i$  for which $d_i = m-1$, and by Theorem 16 of \cite{ChCoTr16} this happens exactly when $G$ is  an NG-1 graph.

Similarly,  the last part of $\alpha(\pi)$  is at least 2  precisely when there is a box in $F(\pi)$ in row $m-1$, column $m$.  This occurs if and only if $d_{m-1} \ge m$ or equivalently when $m$ is the smallest index $i$ for which $d_i = m-1$, and this happens exactly when  $G$ is  an NG-2 graph, again using Theorem 16 of \cite{ChCoTr16}.
\end{proof}

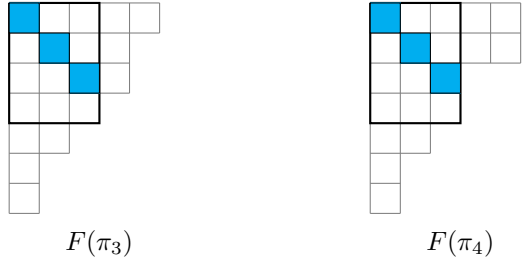
\begin{figure}[h]\begin{center} 
\begin{tikzpicture}[scale=.8]


\draw[gray] (0,0)--(0,3.5);
\draw[gray]  (.5,0)--(.5,3.5);
\draw[gray]  (1,1)--(1,3.5);
\draw[gray]  (1.5,1.5)--(1.5,3.5);
\draw[gray]  (2,2)--(2,3.5);
\draw[gray]  (2.5,3)--(2.5,3.5);

\draw[gray]  (0,3.5)--(2.5,3.5);
\draw[gray]  (0,3)--(2.5,3);
\draw[gray]  (0,2.5)--(2,2.5);
\draw[gray]  (0,2)--(2,2);
\draw[gray]  (0,1.5)--(1.5,1.5);
\draw[gray]  (0,1)--(1,1);
\draw[gray]  (0,0)--(.5,0);
\draw[gray]  (0,.5)--(.5,.5);

\draw[thick, black] (0,3.5)--(1.5,3.5)--(1.5,1.5)--(0,1.5)--(0,3.5);
\draw [fill=cyan] (0,3) rectangle (.5,3.5);
\draw [fill=cyan] (.5,2.5) rectangle (1,3);
\draw [fill=cyan] (1,2) rectangle (1.5,2.5);

\node(0) at (1.5,-.5) {$F(\pi_3)$};

\draw[gray] (6,0)--(6,3.5);
\draw[gray]  (6.5,0)--(6.5,3.5);
\draw[gray]  (7,1)--(7,3.5);
\draw[gray]  (7.5,1.5)--(7.5,3.5);
\draw[gray]  (8,2.5)--(8,3.5);
\draw[gray]  (8.5,2.5)--(8.5,3.5);

\draw[gray]  (6,3.5)--(8.5,3.5);
\draw[gray]  (6,3)--(8.5,3);
\draw[gray]  (6,2.5)--(8.5,2.5);
\draw[gray]  (6,2)--(7.5,2);
\draw[gray]  (6,1.5)--(7.5,1.5);
\draw[gray]  (6,1)--(7,1);
\draw[gray]  (6,0)--(6.5,0);
\draw[gray]  (6,.5)--(6.5,.5);

\draw[thick, black] (6,3.5)--(7.5,3.5)--(7.5,1.5)--(6,1.5)--(6,3.5);
\draw [fill=cyan] (6,3) rectangle (6.5,3.5);
\draw [fill=cyan] (6.5,2.5) rectangle (7,3);
\draw [fill=cyan] (7,2) rectangle (7.5,2.5);

\node(0) at (7.5,-.5) {$F(\pi_4)$};

\end{tikzpicture}\end{center}

 \caption{The Ferrers diagrams   for $\pi_3 = (5,4,4,3,2,1,1)$ and $\pi_4 = (5,5,3,3,2,1,1)$. }

  \label{fig-5}
\end{figure}

The following examples   illustrate Theorem~\ref{NG-block-thm} and its proof.   In Figure~\ref{fig-central-rect},   $\pi_1 = (6,5,2,2,2,1,1,1)$,
 $[\alpha(\pi_1)|\beta(\pi_1)] = [6,4 \, | \,  7,3]$, and a graph with degree sequence $\pi_1$ is NG-2 but not NG-1.
In Figure~\ref{fig-5},   $\pi_3 = (5,4,4,3,2,1,1)$,
 $[\alpha(\pi_3)|\beta(\pi_3)] = [5,3,2 \, | \, 6,3,1]$ and a graph with degree sequence $\pi_3$ is  both NG-1 and NG-2.
Again in  Figure~\ref{fig-5},  $\pi_4 = (5,5,3,3,2,1,1)$,
  $[\alpha(\pi_4)|\beta(\pi_4)] = [5,4,1\, | \, 6,3,1]$ and a graph with degree sequence $\pi_4$ is   NG-1 but not  NG-2.


The next lemma follows from the definition of majorization and is useful in proving results about NG-1 and NG-2 $S$-blocks.

  \begin{lemma}
   Suppose 
 $\gamma_1, \gamma_2 \in Dis_k(n)$  with  $\gamma_1 = (e_1, e_2, \ldots, e_k)$, $\gamma_2 = (f_1, f_2, \ldots, f_k)$, and $\gamma_2 \succeq \gamma_1$.  If     $e_k=1$ then $f_k=1$, and if $f_k \ge 2$ then $e_k \ge 2$.
\label{NG-one-lem}
\end{lemma}

Corollary~\ref{NG-cor}
appears in \cite{ChCoTr16} with a proof that relies directly on the definition of NG-graphs.  Below we present a new and short proof that takes advantage of  Theorem~\ref{NG-block-thm}.
 
\begin{Cor}
Every unbalanced split graph is NG-1 or NG-2 or both.
\label{NG-cor}
 \end{Cor}
 
 \begin{proof}
 Let $\ab$ be an unbalanced $S$-block, thus $\beta \succeq \alpha$.  If $\ab$ is not NG-2 then by Theorem~\ref{NG-block-thm}, the last part of $\alpha$ is 1.  Now applying Lemma~\ref{NG-one-lem}, we know that the last part of $\beta$ is 1, so $\ab$ is NG-1 by Theorem~\ref{NG-block-thm}.
 \end{proof}

From Corollary~\ref{NG-cor} we know that every unbalanced $S$-block is NG-1 or  NG-2 or both.   This means that the amphora $A(n,k)$  can be divided into sets of  $S$-blocks that are NG-1,   NG-2, or both.  We next show that these sets themselves are amphoras consisting of  the elements majorizing   $\tkntknpr$
and majorized by      elements  coming from $(Max)^n_{k}.$

\begin{Def}
{\rm 
For $n \ge \binom{k+1}{2}$, define  $NG_1(n,k)$   to be the subposet of 
  $  A(n,k)$ consisting of NG-1   $S$-blocks and similarly define $NG_2(n,k)$.  Further, define $NG^\ast_1(n,k)$ to be the  subposet of $NG_1(n,k)$ consisting of those   $S$-blocks that are not  in $NG_2(n,k)$, and similarly define $NG^*_2(n,k)$.
}
\label{Ak-ng1-def}
\end{Def}

We note that some of these sets may be empty.  For example, the only $S$-block in  $A(10,4)$ is $[4321|4321]$, so $NG_2(10,4) = \emptyset$.    Similarly, $NG_1(10,1) = \emptyset$.    In Proposition~\ref{empty-NG-prop}
we determine exactly which of these sets are nonempty.

\begin{Prop}

Let $n,k$ be integers with $n \ge 2$, $k \ge 1$, and $n \ge \binom{k+1}{2}$.  
\medskip

If $k=1$ then $A(n,k)$ consists of a single $S$-block which is in $NG$-2.

\smallskip
If $k \ge 2$ and $n \le \binom{k+2}{2} - 2$, then $A(n,k)$ consists entirely of $NG$-1 $S$-blocks.

\smallskip
If $k \ge 2$ and $n \ge \binom{k+2}{2} - 1$, then $A(n,k)$ contains $S$-blocks  in each of the following sets: $NG^\ast_1(n,k) $, $NG^\ast_2(n,k) $, and 
$NG_1(n,k) \cap NG_2(n,k)$.
\end{Prop}

\begin{proof}
For $k=1$ the only $S$-block in $A(n,k)$ is $[n|n]$ which is $NG$-2 by Theorem~\ref{NG-block-thm}.
For $k \ge 2$ the result follows from Remark~\ref{end-in-one-rem} and the observations that $[\tkn|\tkn] \in NG^\ast_2(n,k) $, $[\tknpr|\tknpr] \in NG^\ast_1(n,k) $,  and 
$[\tkn|\tknpr] \in NG_1(n,k) \cap NG_2(n,k)$.   
\label{empty-NG-prop}
\end{proof}

\begin{Thm}
The subposet $NG_1(n,k)$ of $S$-$Block(n)$  is an amphora  consisting of those $S$-blocks that majorize $\tkntknpr$
and are majorized by some NG-1 element of $(Max)^n_{k}.$  Similarly,   $NG_2(n,k)$ is an amphora  consisting of those $S$-blocks that majorize $\tkntknpr$
and are majorized by some NG-2 element of $(Max)^n_{k}.$
\label{Ak-ng1-thm}
\end{Thm}

\begin{proof}
First consider $\ab \in NG_1(n,k)$. By Theorem~\ref{Ak-thm}, $\ab$  majorizes
$\tkntknpr$, and is majorized by $[\beta|\beta]$ in $(Max)^n_k$. 

Conversely, let $\ab$ be in $S$-$Block(n)$ such that it majorizes $\tkntknpr$ and is
majorized by an NG-1 element  $[\alpha'|\beta'] \in (Max)^n_k$. Then, by Definition~\ref{Ak-def}, $\ab\in A(n,k)$ and by Theorem~\ref{NG-block-thm}, the last part of $\beta'$ is 1.   Now
by Definition~\ref{block-maj-def} and Lemma~\ref{NG-one-lem}, 
  the last part of $\beta$ is 1, and so $\ab \in NG_1(n,k)$. 

The proof for $NG_2(n,k)$ is analogous.
 \end{proof}

 \begin{Thm}
 The intersection $NG_1(n,k) \cap NG_2(n,k)$  is an amphora consisting of those  $S$-blocks that majorize $\tkntknpr$
and are majorized by an $S$-block $[\alpha|\beta]$  where $\alpha,\beta \in Dis_k(n)$, the last part of $\alpha$ is at least 2, the last part of $\beta$ equals 1, and $\beta  \gtrdot \alpha$.
\label{NG-intersect-thm}
 \end{Thm}

\begin{proof}
Let $M$ be the set of $S$-blocks $[ \alpha'|\beta']$ for which $\alpha', \beta' \in Dis_k(n)$, the last part of $\alpha'$ is at least 2, the last part of $\beta'$ equals 1, and $\beta'  \gtrdot  \alpha'$.

First consider $\ab \in NG_1(n,k) \cap NG_2(n,k) $.   By Theorem~\ref{Ak-ng1-thm}, 
 $\ab \succeq [\tkn|\tau'_{k-1}(n)]$.  Next we show $\ab$ is majorized by an element of $M$ and that $M$ is an antichain in $S$-$Block(n)$. We know $\beta \succeq \alpha$, so
   in $Dis_k(n)$ there exists a chain of covering relations:  
$\beta = \gamma_1 \gtrdot \gamma_2 \gtrdot \gamma_3 \gtrdot  \cdots \gtrdot \gamma_m = \alpha$.    Since the last part of $\alpha$ is at least 2, and the last part of $\beta$ equals 1,  there exists $r:1 \le r < m$ for which  the last part of $\gamma_r $  is 1 and the last part of $\gamma_{r+1}$ is at least 2.   Thus $[\gamma_{r+1}|\gamma_r] \in M$ and $\ab \preceq [\gamma_{r+1}|\gamma_r] $.  Now to show that $M$ is an antichain, suppose for a contradiction that $\abone, \abtwo \in M$ with $\abone \prec \abtwo$.   Then $\alpha_1 \preceq \alpha_2 \preceq \beta_2 \preceq \beta_1$ and $\alpha_2 \neq \beta_2$ since the last part of  $\alpha_2$ is at least 2 and the last part of $\beta_2$ equals 1.  However, $\beta_1 \gtrdot \alpha_1$, so  $\alpha_1 = \alpha_2$ and $\beta_1 = \beta_2$,  contradicting $\abone \prec \abtwo$.   

 Conversely, let $\ab$ be an $S$-block that majorizes $[\tkn|\tau'_{k-1}(n)]$ and is majorized by an element  $[\alpha'|\beta']$ of $M$.  
 Then $\tkntknpr \preceq \ab \le [\alpha'|\beta'] \preceq  [\alpha'|\alpha']$ so $\ab \in A(n,k)$.  Since $\beta \succeq \beta'$ and $\abpr \in M$, Lemma~\ref{NG-one-lem}
 implies that the last part of $\beta$ is 1, so $\ab \in NG_1(n,k)$ by Theorem~\ref{NG-block-thm}.
   Similarly,  $\alpha \preceq \alpha'$ and $\abpr \in M$, so Lemma~\ref{NG-one-lem} also implies that the last part of $\alpha$ is at least 2.  Thus 
  $\ab \in NG_2(n,k)$  and by Theorem~\ref{NG-block-thm},
and we conclude that  $\ab \in NG_1(n,k) \cap NG_2(n,k)$.
 \end{proof}

\begin{Thm}
The posets of $NG_1^*(n,k)$   and  $NG_2^*(n,k)$  are  amphoras.
Moreover,
 there is a bijection  that preserves majorization between the    first poset  and the amphora  $A(n-k, k-1)$ and one between the second poset and the amphora $A(n-k, k)$.

\label{Ak-ng-only-prop}
\end{Thm}

\begin{proof} 
For any $S$-block $\ab$ in $NG^*_1(n,k)$, the last part of $\alpha$ and the last part of $\beta$ will be 1 by Theorem~\ref{NG-block-thm}.  Similarly, for any  $\ab$ in $NG^*_2(n,k)$, the last part of $\alpha$ and the last part of $\beta$ will be at least 2 by Theorem~\ref{NG-block-thm}.  
Now, it is not hard to see that the following are bijections that preserve majorization.    In the first case,  for each    $ \ab \in NG^*_1(n,k)$, eliminate the $k$th part  of $\alpha$ and the $k$th part of $\beta$, each of    which is 1, and subtract 1 from each of the other parts.  This results in an $S$-block in $A(n-k, k-1)$.
 In the second case,  for each  $\ab \in NG^*_2(n,k)$,  each part  of $\alpha$ and each part of $\beta$ is at least 2, so subtract 1  from each part; this results in an $S$-block in $A(n-k, k)$.
\end{proof}

 In Theorem~\ref{below-balanced-prop} we showed that the property of being unbalanced is upward closed, and we next show that within the class of unbalanced $S$-blocks,   the properties of being NG-1 and NG-2 are downward closed.

\begin{Thm}
Let $\abone, \abtwo $  be unbalanced elements of $S$-$Block(n)$ with  $\abone \succeq \abtwo$.  If $\abone$ is NG-1 then  $\abtwo$ is also NG-1.
If $\abone$ is NG-2 then  $\abtwo$ is also NG-2.
\label{NG-one-prop}
\end{Thm}

\begin{proof}
When $j \neq k$, any $S$-block in $A(n,k)$ is incomparable to any $S$-block in $A(n,j)$ by  Theorem~\ref{incomp-thm}.
Therefore,   $\abone$ and $\abtwo$  both belong to $A(n,k)$  for some $k$ and by Theorem~\ref{Ak-thm}
they both majorize $\tkntknpr$.    Since $\abone$ is  NG-1, by Theorem~\ref{Ak-ng1-thm}, it is majorized by some  NG-1 $S$-block $[\gamma|\gamma] \in (Max)^k_n$.  Then by transitivity, $\abone$ is also  majorized by   $[\gamma|\gamma] \in (Max)^k_n$, thus $\abtwo \in NG_1(n,k)$ and is an NG-1 $S$-block.  The proof of the second part is analogous.
\end{proof}

\section{The split block lattice}

 \label{sec-5}

In this section, we add a maximum and minimum element to the split block poset in order to form a lattice, and determine where the meets and joins of balanced and unbalanced $S$-blocks occur. The  meet operation is not well-defined on the poset $S$-$Block(n)$.  For example,  if  $\ab = [4321 | 721] \wedge [532 | 73]$ then len$(\alpha) \ge 4$ and len$(\beta) \le 2$ by Lemma~\ref{length-lemma} and Definition~\ref{block-maj-def}, and thus $\ab$ is not an $S$-$Block$.  Similarly, it is not possible to take the join of two maximal elements in $S$-$Block(n)$.
 However, $S$-$Block(n)$ can be extended to a lattice  when a maximum element $\hat{1}$   and   a minimum element $\hat{0}$  are added, and in the next theorem, we show that  the meet and join are well-defined.



\begin{Thm}  \label{poset-theorem} The poset
     $S$-$Block(n) \cup \{ \hat{0},\hat{1} \}$ is a lattice.  If $\abone$ and $\abtwo$ are $S$-blocks and $\ell = \max \{ len(\alpha_1) - len(\beta_2), len(\alpha_2) - len(\beta_1)\}$ then the meet and join  of  $\abone$ and $\abtwo$ are given as follows.

{\rm (i)}  If $\ell=0$ or $1$, then $\abone \wedge \abtwo =[\ama \  |  \ \bjb]$, which is an $S$-$block$. Otherwise, $\abone \wedge \abtwo=\hat{0}$. 

{\rm (ii) } If  $ (\bmb) \succeq (\aja)$, then $  \abone \vee \abtwo =  [ \aja \  |  \ \bmb]$, which is an $S$-$block$. Otherwise, $  \abone \vee \abtwo = \hat{1}$.

\label{block-lattice-thm}
\end{Thm}

\begin{proof}
 We begin by proving (i).
Let $\abone$ and $\abtwo$ be $S$-blocks.     Thus $\beta_1 \succeq \alpha_1$ and $\beta_2 \succeq \alpha_2$, so len$(\beta_1) \le  $ len$(\alpha_1)$ and len$(\beta_2) \le  $ len$(\alpha_2)$ by Lemma~\ref{length-lemma}.  We next observe that $\ell \ge 0$ for otherwise, 
len$(\beta_1) \le $ len$(\alpha_1) < len(\beta_2) \le len(\alpha_2) < len(\beta_1)$, a contradiction. 

 If $\ell \ge 2$, we show there is no $S$-block below both $\abone$ and $\abtwo$ in $S$-$Block(n)$.    For a contradiction, suppose $\ell \ge 2$ and there exists an $S$-block $\abhat$ with $\abhat \preceq \abone$ and $\abhat \preceq \abtwo$.  Without loss of generality, we may assume len$(\alpha_1) - $ len$(\beta_2) \ge 2$.    Then $\hat{\alpha} \preceq  \alpha_1$ and $\hat{\beta} \succeq \beta_2$ so 
len$(\hat{\beta}) \le $ len$(\beta_2) \le $ len$(\alpha_1) -2 \le $ len$(\hat{\alpha}) -2$, contradicting the assumption that $\abhat$ is an $S$-block.
Thus if $\ell \not\in \{0,1\}$ then there is no $S$-block below both $\abone$ and $\abtwo$.

Next consider the case that $\ell =0$ or $\ell = 1$.     First we show $[\ama \  |  \ \bjb]$ is an $S$-block and then we show it is the greatest lower bound of $\abone$ and $\abtwo$.
Since $Dis(n)$ is a lattice, we know $(\ama), (\bjb) \in Dis(n)$.  Furthermore,  $(\bjb) \succeq \beta_1 \succeq \alpha_1 \succeq (\ama)$, thus $(\bjb) \succeq (\ama)$.    By Proposition~\ref{mj-length-prop},
len$(\ama) = \max \{ $len$(\alpha_1), $len$(\alpha_2)\}$ and len$(\bjb) = \min \{$len$(\beta_1), $len$(\beta_2)\}$, so $0 \le $len$(\ama) - $len$(\bjb) \le 1$.    Thus $[\ama \  |  \ \bjb]$ satisfies the conditions of Definition~\ref{block-defn},  i.e., is an $S$-block.

We know $[\ama \  |  \ \bjb] $  is a lower bound for $\abone$ and $\abtwo$ and it remains to show it is the greatest lower bound.  Supppose $\abhat \preceq \abone$ and $\abhat \preceq \abtwo$.  The first implies $\hat{\beta} \succeq \beta_1$ and $ \hat{\alpha} \preceq \alpha_1$, and the second implies that  $\hat{\beta} \succeq \beta_2$ and $  \hat{\alpha}\preceq \alpha_2$.  Therefore, $\hat{\beta} \succeq (\bjb)$ and $\hat{\alpha} \preceq (\ama)$, which implies $[\ama \  |  \ \bjb]  \succeq \abhat$.

In proving (ii), note that $(\bmb), (\aja) \in Dis(n)$.  If $(\bmb) \succeq (\aja)$ then len$(\bmb ) \le$ len$ (\aja)$ and len$(\aja) \le $len$(\alpha_1) \le $len$(\beta_1) + 1 \le $len$(\bmb) + 1$, so $ [\aja \ | \bmb]$ satisfies the  length condition of Definition~\ref{block-defn}.
Thus $[\aja \ | \bmb]$ is an $S$-block if and only if $(\bmb) \succeq (\aja)$.  The remainder of the proof is analogous to (i).
\end{proof}


Recall that the unbalanced $S$-blocks are in   amphoras of the form  $A(n,k)$ (Definition~\ref{Ak-def})
and the balanced $S$-blocks are in  amphoras of the form $A(n,k,k-1)$ (Definition~\ref{Akk-def}).
Tables~\ref{table-of-results-bal-unbal} and \ref{table-of-results-NG} show how the properties balanced, unbalanced, NG-1 and NG-2 are passed down by the meet and join operations.  The results are proven in Theorem~\ref{unbal-bal-meet-join-thm} and Theorem~\ref{NG-meet-join-thm}, and examples of each possibility can be found in $S$-$Block(10)$ illustrated in Figure~\ref{fig-s-10}.   Since amphoras have a (unique) minimum element but can have multiple maximal elements, the results for the meet of two $S$-blocks is simpler than  those for  the join.

\begin{table}[h]  \label{t-one} \small
{\renewcommand{\arraystretch}{1.5}
\begin{tabular}{| lcccc  | } \hline
 & $\;\; \abone\;\;$ & $\;\;\abtwo\;\;\;$   & $\abone  \wedge \abtwo$  &  $ \abone  \vee \abtwo$  \\ \hline

1.  & $ A(n,k)$ & $A(n,k)$ &  $A(n,k)$ &   $A(n,k)$  or  $\hat{1}$  \\
2.  &  $A(n,k,k\!-\!1)$& $A(n,k,k\!-\!1)$ &  $A(n,k,k\!-\!1)$  & $A(n,k,k\!-\!1)$  or  $\hat{1}$  \\
3a.  $k> j\!+\!1$ & $A(n,k)$ & $A(n,j)$ &      $\hat{0}$   & $\hat{1}$\\
3b.  $k=j\!+\!1$ & $A(n,k)$ & $A(n,k\!-\!1)$ &   $A(n,k,k\!-\!1)$     & $\hat{1}$\\
4a. $k >j\! +\!1$ & $A(n,k,k\!-\!1)$ & $A(n,j,j\!-\!1)$  & $\hat{0}$ &    $\hat{1}$ \\
4b. $k=j\!+\!1$ & $A(n,k,k\!-\!1)$ & $A(n,k\!-\!1,k\!-\!2)$  & $\hat{0}$ &  $A(n,k\!-\!1)$ or $\hat{1}$ \\
  5. $k\! \in\!\{j, j\!-\!1\}$ & $A(n,k)$ &$A(n,j,j\!-\!1)$  & $A(n,j,j\!-\!1)$ &    $A(n,k)$ or $\hat{1}$ \\
6.  $k \!\not\in\!\{j,j\!-\!1\}$ &$A(n,k)$ & $ A(n,j,j\!-\!1)$  & $\hat{0}$ &    $\hat{1}$ \\

\hline
\end{tabular}
}
\normalsize

\caption{ When  len$(\alpha_1) = k$,  len$(\alpha_2) = j$, and  $\abone$ and $\abtwo$ belong to the amphoras specified in columns 2 and 3,  their meet and join are specified in columns 4 and 5.}  
\label{table-of-results-bal-unbal}
\end{table}
 
\begin{Thm}
Let $\abone, \abtwo \in S$-$Block(n) $, with $k = len(\alpha_1)$ and $j = len(\alpha_2)$. 
The results  for $\abone  \wedge \abtwo$ and $\abone  \vee \abtwo$ shown in Table~\ref{table-of-results-bal-unbal} are correct and in each case the join  is $  \hat{1}$ precisely when  $(\beta_1 \wedge \beta_2) \not\succeq ( \alpha_1 \vee \alpha_2) $.

\label{unbal-bal-meet-join-thm}
\end{Thm}

\begin{proof}

First we show the results for $\abone \wedge \abtwo$.
If $\abone$ and $\abtwo$ are elements of an amphora $A$  then the minimum element of $A$ is a lower bound for $\abone$ and $\abtwo$.  Thus the greatest lower bound is also in $A$  by Definition~\ref{amphora-def}.  This immediately proves cases (1) and (2).    It also proves cases (3b) and (5) when we apply Theorem~\ref{incomp-thm} and  consider the amphora in Theorem~\ref{big-amphora-thm}.

Now suppose $\abone$ is an element of amphora $A_1$ and  $\abtwo$ is an element of amphora $A_2$, where $A_1, A_2 \in W(n)$ and have no lower bound in $W(n)$.  Then by Theorem~\ref{incomp-thm}, $\abone$ and $\abtwo$ have no lower bound in $S$-$Block(n)$ and consequently, $\abone \wedge \abtwo = \hat{0}$ in $S$-$Block(n)$.  This proves  the remaining cases:  (3a), (4a), (4b), and (6).

It remains to show   the results for $\abone \vee \abtwo$, which we denote by $\abpr$.    Suppose $\abone$ is an element of amphora $A_1$ and  $\abtwo$ is an element of amphora $A_2$, where $A_1, A_2 \in W(n)$.  If there is no upper bound for $A_1$ and $A_2$ in $W(n)$, then by Theorems~\ref{partition-thm} and \ref{incomp-thm}
there is no $S$-block above both $\abone$ and $\abtwo$, and hence  $\abone \vee \abtwo = \hat{1}$.  This proves cases (3a), (3b),  (4a) and (6).

In case (1) we have $\abone, \abtwo \in A(n,k)$ for some $k$. 
 By Proposition~\ref{mj-length-prop}, $\alpha', \beta'\in Dis_k(n)$. Thus $\abpr \in A(n,k)$ if $\beta'\succeq \alpha'$ and otherwise $\abpr=\hat{1}$ by Theorem~\ref{poset-theorem}. The proof in case (2) is analogous.

To prove (4b), let $\abone\in A(n,k,k-1)$ and $ \abtwo \in A(n,k-1,k-2)$. 
In this case,  len$(\alpha_1) = k$, len$(\alpha_2) = k-1$,
len$(\beta_1) = k-1$ and len$(\beta_2) = k-2$.  
By Proposition~\ref{mj-length-prop}, len$(\alpha') = k-1$ and len$(\beta') = k-1$. If  $\beta'\succeq \alpha'$, then by Theorem~\ref{poset-theorem}, $\abpr$ is in $S$-$Block(n)$ and by Definition~\ref{Ak-def}
$\abpr\in A(n,k-1)$.  Otherwise $\abpr=\hat{1}$.

Finally, to prove (5), let $\abone\in A(n,k)$ and $\abtwo\in A(n,j,j-1)$ where $k=j$ or $k=j-1$.  
If $\beta' \not\succeq \alpha'$, then $\abpr$ not an $S$-block and $\abpr = \hat{1}$.
Otherwise,  $\beta' \succeq \alpha'$, so  $\abpr$  is an $S$-block.   Since $\abone \in A(n,k)$ and  $\abone \preceq \abpr$,  Theorem~\ref{incomp-thm} implies that $\abpr \in A(n,k)$.  This completes the proof.
\end{proof}

Recall that the class of unbalanced split graphs is the union of the classes of NG-1 and NG-2 graphs.   In Theorem~\ref{NG-meet-join-thm},  we refine the results in Theorem~\ref{unbal-bal-meet-join-thm}  in the case that $k=j$.  The table in Figure~\ref{table-of-results-NG} summarizes the results.

\begin{table}[h]
\resizebox{0.9\textwidth}{!}{\begin{minipage}{\textwidth}
{\renewcommand{\arraystretch}{1.5}

 \begin{tabular}{| c c c c   c  | } \hline
 $i$ &  $ \abone$ & $\abtwo$   & $\abone  \wedge \abtwo$  &  $ \abone  \vee \abtwo$ \\ \hline
  $1,2$& $NG_i(n,k)$ & $NG_i(n,k)$ &   $NG_i(n,k)$   & $NG_i(n,k)$  or  $\hat{1}$  \\
  $1,2$&$NG^*_i(n,k)$ & $NG^*_i(n,k)$ &   $NG^*_i(n,k)$  & $NG^*_i(n,k)$ or  $\hat{1}$  \\
  &$NG^*_1(n,k)$ & $NG^*_2(n,k)$ & $NG_1(n,k) \cap NG_2(n,k)$ &      $\hat{1}$   \\


\hline
\end{tabular} }

 \end{minipage}
 }


\caption{For $i = 1$ or $2$,  when   $\abone$ and $\abtwo$ belong to the amphoras specified in columns 2 and 3,  then their meet and join are specified in columns 4 and 5.}
 \label{table-of-results-NG}
\end{table}

Before proving Theorem~\ref{NG-meet-join-thm}, we introduce several technical lemmas  that help us determine when the meet and join of two partitions are  NG-1 or NG-2.
  
  \begin{lemma}
  Suppose $\gamma_1,\gamma_2 \in Dis_k(n)$, 
  with $\gamma_1 = (a_1, a_2, \ldots, a_k)$ and $\gamma_2 = (b_1, b_2, \ldots, b_k)$ and $a_k = b_k = 1$.    
  Then the last part of $\gamma_1 \vee \gamma_2$ is $1$ and the last part of $\gamma_1 \wedge \gamma_2$ is $1$. 
\label{last-part-lem}
  \end{lemma}
  
  \begin{proof} 
  Let $\alpha = \gamma_1 \vee \gamma_2$ and $\beta = \gamma_1 \wedge \gamma_2$.    By Proposition~\ref{mj-length-prop}, we know  $ \alpha, \beta \in Dis_k(n)$.  By Lemma~\ref{NG-one-lem}, since $\alpha\succeq \gamma_1$ and $\alpha, \gamma_1\in Dis_k(n)$, then the last part of $\alpha $ is 1. 
  
  Let $\hat{\gamma_1} = (a_1-1, a_2-1, \ldots, a_{k-1}-1)$ and $\hat{\gamma_2} = (b_1-1, b_2-1, \ldots, b_{k-1}-1)$ and note that $\hat{\gamma_1}, \hat{\gamma_2} \in Dis_{k-1}(n-k)$.  Let $\hat{\beta} = \hat{\gamma_1} \wedge \hat{\gamma_2}$. By Proposition~\ref{mj-length-prop}, $\hat{\beta} \in Dis_{k-1}(n-k)$.  Write $\hat{\beta} = (d_1, d_2, \ldots, d_{k-1})$, and let 
 $ \tilde{\beta}=(d_1+1, d_2+1 \ldots, d_{k-1}+1,1)$.  
Since adding 1 to each distinct part preserves distinctness, then $\tilde{\beta}\in Dis_k(n)$. It is straightforward to see that  $\gamma_1,\gamma_2\succeq \tilde{\beta}$. By the property of greatest lower bound, $\beta\succeq \tilde{\beta}$, and by Lemma~\ref{NG-one-lem}, the last part of $\beta $ is~1. 
  \end{proof}
  
  \begin{lemma}
  Suppose $\gamma_1,\gamma_2 \in Dis_k(n)$, 
  with $\gamma_1 = (a_1, a_2, \ldots, a_k)$ and $\gamma_2 = (b_1, b_2, \ldots, b_k)$ and $a_k \ge b_k \ge 2$.    
  Then the last part of $\gamma_1 \vee \gamma_2$ is at least 2 and the last part of $\gamma_1 \wedge \gamma_2$ is at least 2.
  \label{last-part-lem-2}
  \end{lemma}
  
  \begin{proof} 
  Let $\alpha = \gamma_1 \vee \gamma_2$ and $\beta = \gamma_1 \wedge \gamma_2$.    By Proposition~\ref{mj-length-prop}, we know  $ \alpha, \beta \in Dis_k(n)$.  By Lemma~\ref{NG-one-lem}, since $\beta\preceq \gamma_1$ and $\beta, \gamma_1\in Dis_k(n)$, then the last part of $\beta $ is at least 2. 
  
  Let $\hat{\gamma_1} = \gamma_1 - (b_k-1)(1,1,\ldots, 1) =  (a_1-b_k + 1, a_2-b_k + 1, \ldots, a_{k}-b_k + 1)$ and $\hat{\gamma_2} = \gamma_2 - (b_k-1)(1,1,\ldots, 1) = (b_1-b_k + 1, b_2-b_k + 1, \ldots, 1)$ and note that $\hat{\gamma_1}, \hat{\gamma_2} \in Dis_{k}(n-k(b_k-1))$.  Let $\hat{\alpha} = \hat{\gamma_1} \vee \hat{\gamma_2}$. By Proposition~\ref{mj-length-prop}, $\hat{\alpha} \in Dis_{k}(n-k(b_k-1))$.  Write $\hat{\alpha} = (d_1, d_2, \ldots, d_{k})$, and let 
 $ \tilde{\alpha}= \hat{\alpha} + (b_k -1)(1,1, \ldots, 1) = (d_1+b_k - 1, d_2+b_k - 1 \ldots, d_{k} + b_k - 1)$, and note that $d_k + b_k -1 \ge 2$.  
Since adding $b_k-k+1$ to each distinct part preserves distinctness, then $\tilde{\alpha}\in Dis_k(n)$. It is straightforward to see that  $\gamma_1,\gamma_2\preceq \tilde{\alpha}$. By the property of the least upper  bound, $\alpha\preceq \tilde{\alpha}$, and by Lemma~\ref{NG-one-lem}, the last part of $\alpha $ is at least 2. 
  \end{proof}

\begin{Thm}
The results  for $\abone  \wedge \abtwo$ and $\abone  \vee \abtwo$ shown in Table~\ref{table-of-results-NG} are correct and the latter value is $  \hat{1}$ precisely when  $(\beta_1 \wedge \beta_2) \not\succeq ( \alpha_1 \vee \alpha_2) $.

\label{NG-meet-join-thm}
\end{Thm}

\begin{proof}
    
We first prove the results for $\abone \wedge \abtwo$ in the first three rows of the table.
By Theorems~\ref{Ak-ng1-thm} and \ref{NG-intersect-thm}, we know that $NG_i(n,k)$  and  $NG^*_i(n,k)$ are amphoras for  $i =1,2$, so as in the proof of  Theorem~\ref{unbal-bal-meet-join-thm}, this proves the results for $\abone \wedge \abtwo$ in rows 1 and 2 of the table.  In the third row, 
$\abone \wedge \abtwo  \in NG_i(n,k)$ because $NG_i(n,k)$ is an amphora, and then the result follows from Proposition~\ref{NG-one-prop}.

We next prove the results for the join in the first three rows of the table, and  in all remaining cases we denote $\abone \vee \abtwo$ by $\abpr$.
First  let $\abone$ and $\abtwo$  be elements of  $NG_1(n,k)$.     By Theorem~\ref{NG-block-thm}, the last part of $\beta_1$ is 1 (and likewise for $\beta_2$).   
By Lemma~\ref{last-part-lem}, the last part of $\beta'$ is 1.  Thus  if $\beta' \succeq \alpha'$, then $\abpr \in NG_1(n,k)$, and otherwise, $\abpr = \hat{1}$.  The result   for  $NG_2(n,k)$ is analogous and uses Lemma~\ref{last-part-lem-2}
to conclude that the last part  of $\alpha'$ is at least $2$. This completes the join result for row 1 of the table.

In row 2 of the table,  let  $\abone$ and $\abtwo$  be elements of  $NG^*_1(n,k)$.  
By Theorem~\ref{NG-block-thm}, the last part of each of $\alpha_1$, $\beta_1$, $\alpha_2$, and $\beta_2$ is 1, so by 
 Lemma~\ref{last-part-lem},   the last part of  $\alpha'$ and of $\beta'$ is 1.  Thus  if $\beta' \succeq \alpha'$, then $\abpr$ is  in $NG^*_1(n,k)$, and otherwise, $\abpr = \hat{1}$.  The result   for elements of $NG^*_2(n,k)$ is analogous and uses Lemma~\ref{last-part-lem-2}
to conclude that the last part   of  each of $\alpha'$  and $\beta'$ is at least $2$.

In row 3 of the table,  let $\abone \in NG^*_i(n,k)$ and $\abtwo \in NG_1(n,k) \cap NG_2(n,k)$.  Since  $\abone$ and $\abtwo$ are both elements of $A(n,k)$, Theorem~\ref{unbal-bal-meet-join-thm} implies that $\abpr \in A(n,k)$ when $\beta' \succeq \alpha'$ and $\abpr = \hat{1}$ otherwise.  In the former case, we know $\abpr \in NG^*_i(n,k)$ by Proposition~\ref{NG-one-prop}.

Finally, we prove the meet and join results for row 4 of the table.  We have $\abone, \abtwo \in A(n,k)$, so     Theorem~\ref{unbal-bal-meet-join-thm} implies that $\abone \wedge \abtwo \in A(n,k)$.  Furthermore,  Theorem~\ref{NG-one-prop} tells us that  each of the classes of NG-1 and NG-2 graphs are downward closed within unbalanced $S$-blocks, therefore 
$\abone \wedge \abtwo$ is  in $NG_1(n,k) \cap NG_2(n,k)$.    For $\abpr$, we know $(\alpha_1 \vee \alpha_2) \succeq \alpha_1$, and  the last part of $\alpha_1$ equals 1, so by Lemma~\ref{NG-one-lem}, the last part of $(\alpha_1 \vee \alpha_2) $ equals 1.  Similarly,   $(\beta_1 \wedge \beta_2) \preceq \beta_2$, and  the last part of $\beta_2$  is at least 2, so by Lemma~\ref{NG-one-lem}, the last part of $(\beta_1 \wedge \beta_2) $  is at least 2.  Now again using  Lemma~\ref{NG-one-lem}, we know $(\beta_1 \wedge \beta_2)  \not\succeq (\alpha_1 \vee \alpha_2) $, so $\abpr = \hat{1}$.
\end{proof}

\section{Pseudo-split graphs and NG-3 graphs}
 \label{sec-6}
 
Pseudo-split graphs and NG-3 graphs were introduced in  Section~\ref{sec-2} and now we study blocks $\ab$ that arise from these two families of graphs.   Recall from Definition~\ref{block-defn}
that $\beta \succeq_w \alpha$ for  blocks $\ab$ and that $\alpha$ and $\beta$ are partitions of integers $n_1$ and $n_2$ respectively where $n_2 \ge n_1$.

\subsection{Characterization results}

 Both NG-3 graphs and pseudo-split graphs have degree sequence characterizations, and in this section we give characterizations  of these classes based on the $\ab$ form of their degree sequences.  NG-3 graphs are not split graphs because they contain an induced 5-cycle.  However, an NG-3 graph $G$ becomes a split graph when the vertices in $A_G$ are removed, and this facilitates our characterization theorem.  Figure~\ref{fig-NG3} shows the Ferrers diagram for the degree sequence $ (7,6,4,4,4,4,4,1)$ of an NG-3 graph.  If the ten boxes marked with $\times$s are removed, the result is the Ferrers diagram of a split graph.

\begin{Thm}
Let $G$ be a graph with degree sequence $\pi$.  Graph $G$ is an NG-3 graph  if and only if the following four conditions all hold:

{\rm (i)}  $\alpha(\pi)$ ends in $2,1$

{\rm (ii)}  $\beta(\pi)$ ends in $4,3$

{\rm (iii)}  len$(\alpha(\pi)) = $ len$(\beta(\pi))$

{\rm (iv)}  $\tbeta \succeq  \talpha$ where $\tbeta$ is $\beta(\pi)$ with the smallest two parts removed and $\talpha$ is $\alpha(\pi)$ with the smallest two parts removed.

\label{NG3-thm}
\end{Thm}

\begin{proof}
First suppose that $G$ is an NG-3 graph,  let $\pi$ be its degree sequence, and let $\pi = (d_1, d_2, \ldots, d_r)$.   Let  $\alpha = \alpha(\pi)$,  $\beta = \beta(\pi)$, $m = m(\pi)$, and write $V(G) = \{v_1, v_2, \ldots, v_r\}$, where $deg(v_i) = d_i$ for each $i$.  
By Theorem 21 of \cite{ChCoTr16}, $d_i = m-1$ if and only if $i \in \{m-2, m-1, m, m+1, m+2\}$.    Now Remark~\ref{m-ABC-rem} and Definition~\ref{ABC-def}
 imply that the sets 
 $A_G = \{v_{m-2}, v_{m-1}, v_m, v_{m+1}, v_{m+2} \}$, 
$B_G = \{ v_1, v_2, \ldots, v_{m-3}\}$,  
$C_G = \{v_{m+3},v_{m+4}, \ldots, v_r\}$ form the $ABC$-partition of $G$.  Since $d_{m-1} = d_m = m-1$, we know 
that len$(\alpha) = $ len$(\beta) = m-1$.   In addition, $d_{m-2} = m-1$, so 
 the smallest two parts of $\alpha$ are  $2,1$.  Finally, using Theorem~\ref{NG-charac},
 there are no edges in $G$ between vertices in $A_G$ and vertices in $C_G$, so $d_{m+3} = deg(v_{m+3}) \le |B_G| = m-3$, and this means that the 
 smallest  two parts of $\beta$ are $4,3$.   It remains to show (iv).

Removing  the smallest  two parts of $\alpha$ and the smallest  two parts of $\beta$ corresponds to removing a $5 \times 2$ rectangle from $F(\pi)$, consisting of the boxes in 
the rows $m-2$ through $m+2$ that are in columns $m-2$ and $m-1$.
 This  transforms the degree sequence $\pi$ by removing 2 from each of  $d_{m-2}, d_{m-1}, d_m, d_{m+1}, d_{m+2}$,  corresponding to the five vertices in the set $A_G$.  The resulting degree sequence $\tpi$  is the degree sequence of a split graph, since removing the edges in $A_G$  (a 5-cycle) gives a split graph $H$ with  $KS$-partition as follows:    $K = B_G$ and $S = A_G \cup C_G$.  Thus $\tbeta \succeq \talpha$ by Theorem~\ref{split-S-block-thm}, and this establishes (iv) and completes the  proof of the forward direction.

To prove the converse,  we assume that graph $G$ with degree sequence $\pi$ satisfies the four conditions and show $G$ is an NG-3 graph.  It will suffice to show that $G$ has the same degree sequence as an NG-3 graph, since membership in the class of NG-3 graphs can be recognized  by degree sequences \cite{ChCoTr16}.

Let  $ (d_1, d_2, \ldots, d_r)$ be the degree sequence $\pi$ and let  $m = m(\pi)$.  Let  $\alpha = \alpha(G)$,  $\beta = \beta(G)$,  and  $k = len(\alpha)$.    By condition (iii), we have len$(\beta)= k$.  Since $\alpha$ ends in $2,1$, and $\beta$ ends in $4,3$,  we know $d_{k-1} = d_k = d_{k+1} = d_{k+2} = d_{k+3} = k$ and if $ r \ge k+4 $ then $d_{k+4} \le k-2$ (for example, see Figure~\ref{fig-NG3}).  By the definition of $m(\pi)$ we know $m = k+1$.  
Let $\tpi$ be the sequence $ \tilde{d_1}, \tilde{d_2},  \tilde{d_3},  \ldots , \tilde{d}_{r}$  defined as follows:
$\tilde{d_i} = d_i -2$ for $k-1 \le i \le k+3$ and otherwise, $\tilde{d_i} = d_i$.   	  Since $\pi$ is non-increasing  and $d_{k+4} \le k-2 \le d_{k+3} - 2$  (if $r \ge k+4$) we conclude that $\tpi$ is also  non-increasing.
  The Ferrers diagram $F(\tpi)$ is obtained from $F(\pi)$ be removing the $5 \times 2$ rectangle consisting of the boxes in rows $k-1$ through $k+4$ that are in columns $k$ and $k-1$.  As a result, boxes on the diagonal in rows $k$ and $k-1$ are removed from $F(\pi)$ and len$(\alpha(\tpi)) = $ len$(\beta(\tpi)) = k-2$.

  By construction,  $\tbeta = \beta(\tpi)$ and $\talpha = \alpha(\tpi)$, and     by hypothesis, $\tbeta \succeq \talpha$.  Thus Theorem~\ref{split-S-block-thm}  and Definition~\ref{block-defn}
  imply that  $\tpi$ is the degree sequence of a split graph.    Let $\tg$ be a split graph with degree sequence $\tpi$ with vertex set $v_1, v_2, \cdots, v_{r}$, where $deg(v_i) = \tilde{d_i} $.
  In $\tpi$ we have  $\tilde{d}_{k-1} = k-2 = (k-1) - 1$ and $\tilde{d_{k}} = k-2 < k - 1$, so $m(\tg) = k-1 =  m-2$. 
  
  The partition $\tk = \{v_1, v_2,  \ldots, v_{k-1}\}$ and $\ts = \{v_{k}, v_{k+1}, \ldots, v_{r}\}$   is a $K$-max partition of $\tg$ (see, for example,
  Remark 14 of \cite{ChCoTr16}).  Since vertex $v_{k-1}$ has degree $k-2$ in $\tg$, all of its neighbors are in $K$ and thus $\{v_{k-1}, v_{k},  v_{k+1},  v_{k+2},  v_{k+3} \}$ forms a stable set in $\tg$.  
  Starting with graph $\tg$, add five edges to form a $5$-cycle among the vertices $v_{k-1}, v_{k},  v_{k+1},  v_{k+2},  v_{k+3}$ and call the resulting graph $H$.  By construction, $H$ has the same degree sequence as $G$, and furthermore, 
   $H$ is an NG-3 graph  with the following $ABC$-partition:  $A_H = \{v_{k-1}, v_{k},  v_{k+1},  v_{k+2},  v_{k+3}\}$, $B_H = \{v_1, v_2,  \ldots, v_{k-2}\}$, and $C_H = \{v_{k+4}, v_{k+5}, \ldots, v_{r}\}$.   
  Since   membership in the class of NG-3 graphs is completely determined by degree sequence \cite{ChCoTr16}, we conclude that $G$ is an NG-3 graph.
\end{proof}

A more algebraic proof of Theorem~\ref{NG3-thm} can be obtained by using the summation condition of the characterization theorem for NG-3 graphs, which appears in \cite{ChCoTr16}.

We now characterize pseudo-split graphs from the $\ab$ form of their degree sequence.

\begin{Cor}
Let  $G$ be a graph with degree sequence $\pi$.  Graph $G$ is a pseudo-split graph  if  and only if  $\beta(\pi) \succeq \alpha(\pi)$ or the four conditions of Theorem~\ref{NG3-thm}
are satisfied.
\label{pseudo-charac}
\end{Cor}

\begin{proof}
First suppose $G$ is a pseudo-split graph.  By Theorem~\ref{pseudo-split-union},
$G$ is either a split graph or an NG-3 graph.   
 In the first case, 
 Theorem~\ref{split-S-block-thm}   and Definition~\ref{block-defn} imply that $\beta(\pi) \succeq \alpha(\pi)$.  The second case follows from Theorem~\ref{NG3-thm}.  
 
 Conversely, if $\beta(\pi) \succeq \alpha(\pi)$, then by Theorem~\ref{length-alpha}, Definition~\ref{block-defn},  and 
 Theorem~\ref{split-S-block-thm},    we know that $\pi$ is the degree sequence of a split graph.  Since membership in the class of split graphs is determined by degree sequence, $G$ is a split graph.  The second case again follows immediately from Theorem~\ref{NG3-thm}.
\end{proof}


\begin{figure}\begin{center}
\begin{tikzpicture}[scale=.8]


\draw[gray] (0,0)--(0,4);
\draw[gray]  (.5,0)--(.5,4);
\draw[gray]  (1,.5)--(1,4);
\draw[gray]  (1.5,.5)--(1.5,4);
\draw[gray]  (2,.5)--(2,4);
\draw[gray]  (2.5,3)--(2.5,4);
\draw[gray]  (3,3)--(3,4);
\draw[gray]  (3.5,3.5)--(3.5,4);

\draw[gray]  (0,4)--(3.5,4);
\draw[gray]  (0,3.5)--(3.5,3.5);
\draw[gray]  (0,3)--(3,3);
\draw[gray]  (0,2.5)--(2,2.5);
\draw[gray]  (0,2)--(2,2);
\draw[gray]  (0,1.5)--(2,1.5);
\draw[gray]  (0,1)--(2,1);
\draw[gray]  (0,.5)--(2,.5);
\draw[gray]  (0,0)--(.5,0);

\draw[thick, black] (0,4)--(2,4)--(2,1.5)--(0,1.5)--(0,4);
\draw [fill=cyan] (0,3.5) rectangle (.5,4);
\draw [fill=cyan] (.5,3) rectangle (1,3.5);
\draw [fill=cyan] (1,2.5) rectangle (1.5,3);
\draw [fill=cyan] (1.5,2) rectangle (2,2.5);

\node(0) at (1.5,-.5) {$F(\pi_5)$};
\node(0) at (1.75,.75) {$\times$};
\node(0) at (1.75,1.25) {$\times$};
\node(0) at (1.75,1.75) {$\times$};
\node(0) at (1.75,2.25) {$\times$};
\node(0) at (1.75,2.75) {$\times$};
\node(0) at (1.25,.75) {$\times$};
\node(0) at (1.25,1.25) {$\times$};
\node(0) at (1.25,1.75) {$\times$};
\node(0) at (1.25,2.25) {$\times$};
\node(0) at (1.25,2.75) {$\times$};

\end{tikzpicture}\end{center}


  \caption{The Ferrers diagram and central rectangle for an NG-3 graph $G$ with degree sequence $\pi_5 =(7,6,4,4,4,4,4,1)$.  The boxes marked with an $\times$ represent the edges in the 5-cycle induced by $A_G$.  }
  \label{fig-NG3}
\end{figure}
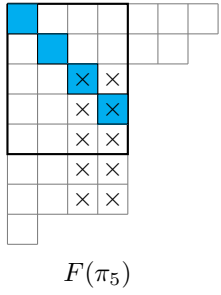

The following example shows that in contrast to Theorem~\ref{NG-one-prop}, the class of NG-3 graphs is neither upward nor downward closed under the majorization of Definition~\ref{block-maj-def}.

\begin{example} {\rm 
Let $\alpha = (10, 8, 6, 2, 1)$, $\beta = (11, 8, 5, 4, 3)$, $\alpha_1 = (10, 8, 5,3, 1)$, and  $\beta_1 = (11, 8, 6, 4, 2)$.  
It is not hard to check that  $[\alpha | \beta_1]$, $ [\alpha | \beta] $, and   $ [\alpha_1 | \beta] $  are all   blocks with $n_1 = 27$ and $n_2 = 31$ and that 
$[\alpha | \beta_1] \succ [\alpha | \beta]  \succ  [\alpha_1 | \beta] $.  
Using Theorem~\ref{NG3-thm}
one can check  that $\ab$ is an NG-3 graph while  $[\alpha_1 | \beta]$  and $ [\alpha | \beta_1]$ are not.  
}
\end{example}

\subsection{The amphora $NG_3(n,k)$}

 In Theorem~\ref{Ak-ng1-thm}
 we show that $NG_1(n,k)$ and $NG_2(n,k)$ are amphoras, and in this section we prove an analogous result for NG-3 graphs.
It follows from Theorem~\ref{NG3-thm}
and its proof that if $\pi$ is the degree sequence of an NG-3 graph and $\alpha(\pi) \in Dis(n_1)$ and $\beta(\pi) \in Dis(n_2)$  then $n_2 = n_1 + 4$.  When $\talpha$ is formed by removing the  parts  $2,1$  from $\alpha(\pi)$ and  $\tbeta$ is formed by removing the  parts $4,3$  from $\beta(\pi)$, then $\talpha$ and $\tbeta$ both partition  $n$ where $ n= n_1 - 3 = n_2 - 7$.  This motivates our definition of the poset $NG_3(n,k)$.

\begin{Def}
{\rm
The elements of poset  $NG_3(n,k)$ are the  NG-3 blocks  $\ab$ for which $\alpha$ has smallest parts $2$, $1$;  $\beta$ has smallest parts $4$, $3$; and the partitions $\talpha$, $\tbeta$ obtained by removing the smallest two parts of $\alpha$ and $\beta$ respectively satisfy the following:

\noindent
(i)  $\talpha, \tbeta  \in Dis_k(n)$,\\
(ii) The smallest part of $\talpha$ is at least $3$, and \\
(iii) The smallest part of $\tbeta$ is at least $5$.

The  poset relation in $NG_3(n,k)$ is the majorization given in Definition~\ref{block-maj-def}

}

\label{Ak-ng3-def}
\end{Def}

 Our first result establishes when $NG_3(n,k)$ is non-empty.

 \begin{Prop}
 The set $NG_3(n,k)$  is non-empty if and only if $n \ge 4k + \binom{k+1}{2}$.
 \label{NG-3-nonempty-prop}
 \end{Prop}
 
 \begin{proof}
 Suppose $\ab \in NG_3(n,k)$ and write $\beta = (b_1, b_2, \ldots, b_k, 4,3)$.  Then $b_k \ge 5$ and  $(b_1, b_2, \ldots, b_k) \in Dis_k(n)$, hence $b_k \ge 5$, 
 $b_{k-1} \ge 6$, $b_{k-2} \ge 7$,  \ldots, $b_{1} \ge k+4$.  So $n = \sum_{i=1}^k b_i \ge 5 + 6 + 7 + \cdots + (k+4) = 4k + (1 + 2 + \cdots + k) = 4k + \binom{k+1}{2}$.
 
 Conversely, if $n \ge 4k + \binom{k+1}{2}$ then write $n = 4k + \binom{k+1}{2} + x$.  For $b_k = 5$, $b_{k-1} = 6$, $b_{k-2} = 7$,  \ldots, $b_{2} = k+3$ and $b'_1 = k+ 4 + x$, we have $(b'_1, b_2, \ldots, b_k) \in Dis_k(n)$ using the above calculation.    Thus $\ab \in  NG_3(n,k)$, where 
 $\alpha = (b'_1, b_2, \ldots, b_k, 2,1)$ and $\beta = (b'_1, b_2, \ldots, b_k, 4,3)$, and $NG_3(n,k)$ is non-empty.
   \end{proof}

 \begin{lemma}
  If $\ab \in NG_3(n,k)$ then the smallest part of $\talpha$ is at least 5.
  \label{at-least-5-lem}
 \end{lemma}
 
 \begin{proof}
 By Definition~\ref{Ak-ng3-def}, we know $\talpha, \tbeta \in Dis_k(n)$.  Let $\talpha = (a_1, a_2, \ldots, a_k)$ and $\tbeta = (b_1, b_2, \ldots, b_k)$, thus $ n =  a_k + \sum_{i=1}^{k-1} a_i =  b_k + \sum_{i=1}^{k-1} b_i  = n$.    Since $\tbeta \succeq \talpha$ we have  $\sum_{i=1}^{k-1} b_i \ge   \sum_{i=1}^{k-1} a_i $ and thus $b_k \le a_k$.  Since $b_k \ge 5  $ in Definition~\ref{Ak-ng3-def}, we conclude that $a_k \ge 5$.
 \end{proof}
 
 For positive integers $n,k,r$, let $Dis_{(k,r)}(n)$ be the set of partitions of $n$ into $k$ distinct parts each of which is greater than $r$.  Define the function
 $f: Dis_{(k,r)}(n) \to Dis_k(n-kr)$, which subtracts $r$ from each part, as follows:  if $(a_1, a_2, \ldots, a_k) \in Dis_{(k,r)}(n)$ with $a_1 > a_2 > \cdots > a_k$ then $f(a_1, a_2, \ldots, a_k) = (a_1-r, a_2-r, \ldots , a_k-r)$.  It is not hard to see that $f$ is well-defined and preserves  majorization, that is, if $\sigma, \gamma \in Dis_{(k,r)}(n)$, then $\sigma \succeq \gamma$ if and only if $f(\sigma) \succeq f(\gamma).$
 
 \begin{Thm}
 Let $n \ge 4k + \binom{k+1}{2}$.   The set $NG_3(n,k)$  is an amphora under majorization.  Moreover,
 there is a bijection  that preserves majorization between    $NG_3(n,k)$  and the amphora  $A(n-4k, k)$ of split blocks.
   \end{Thm}

 \begin{proof}
 By Proposition~\ref{NG-3-nonempty-prop}, the set $NG_3(n,k)$ is non-empty.  For any $\ab \in NG_3(n,k)$, let $\talpha$, $\tbeta$  be obtained by removing the smallest two parts of $\alpha$ and $\beta$ respectively.  We know the smallest two parts of $\talpha$ and $\tbeta$ are at least 5 by Lemma~\ref{at-least-5-lem}
 and Definition~\ref{Ak-ng3-def}, thus $\talpha,\tbeta \in Dis_{(k,4)}(n)$.    The function $f$ defined above with $r=4$ is an order-preserving bijection from $Dis_{(k,4)}(n)$ to $Dis_k(n-4k)$, thus $f(\talpha), f(\tbeta) \in Dis_k(n-4k)$ and $f(\tbeta) \succeq f(\talpha)$.  Hence $[f(\talpha)|f(\tbeta)] \in A(n-4k,k)$ and the resulting bijection $\ab \to [f(\talpha)|f(\tbeta)] $ also preserves majorization by Definition~\ref{block-maj-def}.
 
 \end{proof}

 We illustrate this bijection in the following example.

 \begin{example}{\rm 
 Let $n= 38$ and  $k = 5$, so $n-4k = 20$.  Observe that     $20 + \binom{6}{2} = 35 \le 38$, so $NG_3(38,5)$ is non-empty.    There are three elements in $Dis_5(18)$, namely,  $(8,4,3,2,1) \succeq (7,5,3,2,1) \succeq (6,5,4,2,1)$.   These correspond to adding each of the partitions of 3 to the partition $(5,4,3,2,1)$.  Next we add $4$ to each component to find the possibilities for $\talpha$ and $\tbeta$ as follows: $\gamma_1 \succeq \gamma_2 \succeq \gamma_3$ where $\gamma_1 = (12, 8,7,6,5) $,  $\gamma_2 = (11,9,7,6,5) $ and $\gamma_3 = (10, 9,8,6,5)$.
We   append $2,1$ and $4,3$ to each of these to get the possibilities for $\alpha$ and $\beta$, respectively, and then combine these to get the elements of $NG_3(38,5)$ as shown in Figure~\ref{NG-3-fig}.
}
\label{ng3-ex}
 \end{example}










  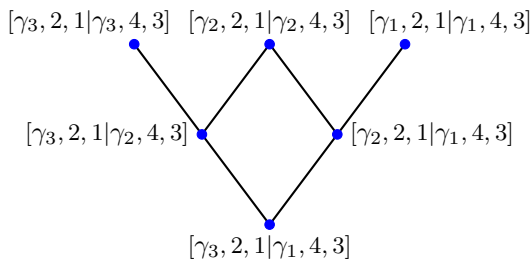
\begin{figure}[ht]
\begin{center}
\begin{tikzpicture}[scale=.6]
\draw[thick] (0,2)--(1.5,0)--(3,2)--(4.5,0)--(6,2);
\draw[thick] (1.5,0)--(3,-2)--(4.5,0);
\filldraw[blue]

(1.5,0) circle [radius=3pt]
(4.5,0) circle [radius=3pt]

(0,2) circle [radius=3pt]
(3,2) circle [radius=3pt]
(6,2) circle [radius=3pt]

(3,-2) circle [radius=3pt]
;

\node(0) at (-1,2.5)  { \small $ [ \mbox{\normalsize $\gamma_3$}, 2, 1|  \mbox{\normalsize $\gamma_3$},4, 3]$};

\node(0) at (3,2.5) {\small$[\mbox{\normalsize $\gamma_2$}, 2, 1|\mbox{\normalsize $\gamma_2$}, 4,3]$};
\node(0) at (7,2.5) {\small $[\mbox{\normalsize $\gamma_1$}, 2, 1|\mbox{\normalsize $\gamma_1$}, 4, 3]$};

\node(0) at (-.6,0) {\small$[\mbox{\normalsize $\gamma_3$}, 2, 1|\mbox{\normalsize $\gamma_2$}, 4, 3]$};
\node(0) at (6.6,0) {\small$[\mbox{\normalsize $\gamma_2$}, 2, 1|\mbox{\normalsize $\gamma_1$}, 4, 3]$};
\node(0) at (3,-2.5) {\small $[\mbox{\normalsize $\gamma_3$}, 2, 1|\mbox{\normalsize $\gamma_1$}, 4, 3]$};

  \end{tikzpicture}
  \end{center}

 
  \caption{The poset $NG_3(38,5)$ from Example \ref{ng3-ex}}
  \label{NG-3-fig}
  \end{figure}


In future work, we expect to see more connections between the structure of graphs, based on their degree sequences, and the structure of lattices of partitions of integers.

\end{document}